\documentclass[12pt,a4paper]{amsart}
\usepackage{amsmath,amsthm,amsfonts,amssymb,amscd,amstext}
\usepackage[dvips]{graphicx}
\usepackage{psfrag,euscript,a4wide}
\usepackage[active]{srcltx}

\usepackage[final]{epsfig}

\theoremstyle{plain}

\newtheorem{theorem}{Theorem}[section]
\newtheorem{proposition}[theorem]{Proposition}
\newtheorem{corollary}[theorem]{Corollary}
\newtheorem{lemma}[theorem]{Lemma}

\theoremstyle{definition}
\newtheorem{remark}[theorem]{Remark}
\newtheorem{definition}[theorem]{Definition}

\newcommand{\RR}{{\mathbb R}}
\newcommand{\NN}{{\mathbb N}}

\newcommand{\cal}{\mathcal}

\renewcommand{\epsilon}{\varepsilon}
\newcommand{\dist}{\operatorname{dist}}

\newcommand{\vari}{\operatorname{var}}

\newcommand{\cF}{\EuScript{F}}

\newcommand{\cP}{\EuScript{P}}

\newcommand{\sgn}{\operatorname{sgn}}

\def \NN {{\mathbb N}}

\def \RR {{\mathbb R}}

\def \cA {{\mathcal A}}

\def \cF {{\mathcal F}}

\def \cP {{\mathcal P}}

\def \cS {{\mathcal S}}

\def \cV {{\mathcal V}}

\title[Lorenz:Decay,Hitting,Recurrence]
{Lorenz like flows: exponential decay of correlations for the Poincar\'e map, logarithm law, quantitative recurrence}

\author{S. Galatolo, M.J. Pacifico}

\date{\today}

\begin{thanks} {  M.J.P. was partially supported by
    CNPq-Brazil/FAPERJ-Brazil/Pronex Dynamical Systems/CRM Ennio De Giorgi-Scuola Normale Superiore-Pisa. }
\end{thanks}

\begin{document}

\address{S. Galatolo, Dipartimento di Matematica Applicata via Buonarroti 1 Pisa} \email{s.galatolo@docenti.ing.unipi.it}
\urladdr{http://www2.ing.unipi.it/~d80288/}

\address{Maria Jos\'e Pacifico, Instituto de Matem\'atica,
Universidade Federal do Rio de Janeiro,
C. P. 68.530, 21.945-970 Rio de Janeiro, Brazil}
\email{pacifico@im.ufrj.br}

\begin{abstract}

In this paper we prove that the Poincar\'e map associated to  a Lorenz like flow has exponential decay of correlations with respect to Lipschitz observables.
This implies that the hitting time associated to the flow satisfies a logarithm law.
The hitting time $\tau _r(x,x_0)$ is the time needed for the orbit of a point $x$ to enter for the first time in a ball $B_r(x_0)$ centered at $x_0$,
with small radius $r$.
As the radius of the ball decreases to $0$ its asymptotic behavior is
a power law whose exponent is related to the local dimension of the SRB measure at $x_0$:
for each $x_0$ such that the local dimension $d_{\mu}(x_0)$ exists,
\begin{equation*}
\lim_{r\rightarrow 0} \frac{\log \tau _r(x,x_0)}{-\log r} =
d_{\mu}(x_0)-1
\end{equation*}
holds for $\mu$ almost each $x$. In a similar way it is possible to consider a quantitative recurrence indicator quantifying the speed of coming back of an orbit to its starting point. Similar results holds for this recurrence indicator.

 \end{abstract}
\maketitle

\setcounter{tocdepth}{2}
\tableofcontents


\section{Introduction}

It is well known that in a chaotic dynamics the pointwise, future  behavior of an initial condition is unpredictable and even impossible to be described by using a finite quantity of information. On the other hand many of its statistical properties  are rather regular and often described by suitable versions of classical theorems from probability theory: law of large numbers, central limit theorem, large deviations estimations, correlation decay, hitting times,  various kind of quantitative recurrence and so on.

In this article we consider a class of flows which contain the
celebrated Geometric Lorenz flow and we will study some of its
statistical features by a sharp estimation for the decay of
correlations of its first return map on a suitable Poincar\'e
section. This will give a quantitative recurrence estimation and an
estimation for the scaling behavior of the time which is needed to hit small targets (logarithm law).

Let $\Phi^t$ be a $C^1$  flow in ${\mathbb R}^3$. Quantitative recurrence estimations and logarithm laws can be seen in the
following framework: we are interested in a quantitative estimation of the
speed of approaching of a certain orbit $\Phi ^{t}(x)$ (starting from the
point $x$) of the system to a given target point $x_{0}$. Let $B_{r}(x_{0})$
be a ball with radius $r$ centered at $x_{0}$. We consider the time
\[
\tau _{r}(x,x_{0})=\inf \{t\in \mathbb{R}^{+}:\Phi ^{t}(x)\in B_{r}(x_{0})\}
\]
needed for the orbit of $x$ to enter in $B_{r}(x_{0})$ for the first time and the asymptotic behavior of $\tau _{r}(x,x_{0})$ as $r$ decreases
to $0$. Often this is a power law of the type $\tau _{r}\sim r^{-d}$ and
then it is interesting to extract the exponent $d$ by looking at the
behavior of
\begin{equation}
R(x,x_{0})=\underset{r\rightarrow 0}{\lim }\frac{\log \tau _{r}(x,x_{0})}{%
-\log r}.
\end{equation}%
In this way, we have a hitting time indicator for  orbits of the
system.\footnote{ Another way to look at the same phenomena is by considering the behavior of the ratio of the distance $\frac{-\log d(\Phi ^{t}(x),x_{0})}{\log t}$ as $t\rightarrow \infty$ (for the equivalence see \cite{GP97}).}

If the orbit $\Phi ^{t}$ starts at $x_{0}$ itself and we consider the
second entrance time in the ball
\begin{equation}\label{eq-2}
\tau _{r}^{\prime }(x_{0})=\inf \{t\in \mathbb{R}^{+}:\Phi ^{t}(x_0)\in
B_{r}(x_{0}),\exists  i<t , s.t.\Phi ^{i}(x_0)\notin B_{r}(x_{0})\}
\end{equation}%
(because the orbit trivially starts inside the ball) with the same
construction as before, we have a quantitative recurrence indicator. If the
dynamics is chaotic enough, often the above indicators converge to a
quantity which is related to the local dimension of the invariant measure of
the system and in the hitting time case this relation is called {\em logarithm
law}.

Hitting time results of this kind (sometime replacing balls with
other suitable target sets) have been proved in many continuous time
dynamical systems of geometrical interest: geodesic flows, unipotent
flows, homogeneous spaces, etc.  (see e.g.
\cite{AM,HV,KM,Su,M,Mas}). For discrete time systems this kind of
results hold in general if the system has fast enough decay of
correlation (\cite{galatolo}). Mixing is however not sufficient,
since this relation does not hold in some slowly mixing system
having particular arithmetical properties (\cite{GP97}). Some
further connections with arithmetical properties are shown in
interesting examples as rotations and interval exchange maps
(see e.g. \cite{G2,kimseo,KM2}). This kind of problem is also
connected with the so called dynamical Borel Cantelli results (see
\cite{GK} and e.g. \cite{T,HV,GK,Dol}). Moreover, in the symbolic
setting, similar results about the hitting time are used in
information theory (see e.g. \cite{Sh,Ko}). About
quantitative recurrence, our approach follows a set of results
connecting a quantitative recurrence estimation with local dimension
(see e.g. \cite{STV,S06,BS,Bo}). We remark that the speed of
correlation decay for Lorenz like flows is not yet
known (although some are proved to be mixing, see \cite{LMP}) hence 
quantitative recurrence and hitting time results cannot be proved
directly using this tool, instead of this we will consider a
Poincar\'e section, estimating its correlation decay and work with
return times.

\subsection{Statement of results}

Let $I=[-\frac{1}{2},\frac{1}{2}]$ be a unit interval, we consider a
flow $X^{t}$ on $\mathbb{R}^{3}$ having a Poincar\'e section on a
square $\Sigma =I\times I$ satisfying the following properties:

\begin{description}
\item[1)] The flow induces\footnote{Up to zero Lebesgue measure sets.} a first return map $F:\Sigma \rightarrow \Sigma $ of the form $F(x,y)=(T(x),G(x,y))$ (preserves the natural vertical foliation of the
square) and:

\item[1.a)] There is $c\in I$ and $k\geq 0$ such that, if $x_{1},x_{2}$ are such that $c\notin \lbrack x_{1},x_{2}]$ then $\forall y\in
I:|G(x_{1},y)-G(x_{2},y)|\leq k\cdot |x_{1}-x_{2}|$

\item[1.b)] $F|_{\gamma }$ is $\lambda $-Lipschitz with $\lambda <1$ (hence
is uniformly contracting) on each vertical leaf $\gamma $:
$|G(x,y_1)-G(x,y_2)|\leq \lambda \cdot |y_{1}-y_{2}|$

\item[1.c)] $T:I\rightarrow I$ is\ onto and piecewise monotonic, with two $%
C^{1}$ increasing branches on the intervals $[-\frac{1}{2},c)$,$(c,\frac{1}{2%
}]$ and $T^{\prime }>1$ where it is defined\footnote{The condition
  $T^{\prime }>1$ can be relaxed to $\lambda [inf_{x\in I}(T'(x))]^{-1}<1  $
  provided that the map $T$ is eventually expanding in the sense of
 \cite{Vi97}, Chapter 3.}. Moreover $\underset{%
x\rightarrow c^{-}}{\lim }T(x)=\frac{1}{2},T(c)=-\frac{1}{2},$
$\underset{x\rightarrow c}{\lim } T^{\prime }(x)=\infty $.

\item[1.d)] $\frac{1}{|T^{\prime }|}$ has bounded variation.
\end{description}

By the statistical properties of the map $T$, which is piecewise expanding,
under the above assumptions, it turns out that $F$ has a unique SRB measure $ \mu _{F}.$ We then ask the following property for the flow:

\begin{description}
\item[2)] The flow $X^{t}$ is transversal to the section $\Sigma $ and its return time to $\Sigma$ is integrable with respect to $\mu_{F}$.
\end{description}

In Section \ref{sec:constr-geometr-model} we will describe the geometric Lorenz system and we show that it satisfies these properties.

The main results of the paper concern  some statistical properties of $X^{t}$
 and $F$, more precisely:

\textbf{Theorem A (decay of correlation for the Poincar\'e map) }\emph{The
unique SRB measure }$\mu _{F}$ \emph{of }$F$\emph{\ has exponential decay of
correlation with respect to Lipschitz observables.}\vspace{0.2cm}

This result is proved in Section \ref{2p1} (Theorem \ref{resuno}) where the reader can also find a precise definition of correlation decay.
 The proof also  uses a regularity estimation for the invariant measure $\mu _{F}$ which can be found in the Appendix I (Lemma 8.1) and is proved by sort of Lasota-Yorke inequality.
 We remark that a stretched-exponential bound for the decay of correlation for a two dimensional Lorenz like map was given in \cite{Bu83} and \cite{ACS}.

We say that a point $x_0 \in {\mathbb R}^3$ is {\em regular} if there are  $y_0 \in \Sigma $ and $t_0\geq 0 $   such that  $X^{t_0}$ induces a diffeomorphism between a neighborhood of $y_0$ and a neighborhood of $x_0$.

In Section \ref{sec:SBRfluxo} we recall how to construct an
SRB ergodic invariant measure for the flow $X^{t}$ which will be
denoted by $\mu _{X}$. It turns out that this measure has the
following property

\textbf{Theorem B (logarithm law for the flow) }\emph{For each regular } $x_{0}$ \emph{
such that the local dimension }$d_{\mu _{X}}(x_{0})$\emph{\ is defined it
holds }%
\begin{equation}
\underset{r\rightarrow 0}{\lim }\frac{\log \tau _{r}(x,x_{0})}{-\log r}%
=d_{\mu _{X}}(x_{0})-1
\end{equation}%
\emph{for a.e. starting point }$x$\emph{.}

This is proved in Section \ref{sec:loglaw} (Theorem \ref{main1}) and uses the above decay of correlation estimation for the first return map $F$, a result from \cite{galatolo} giving the hitting time estimation for  systems having faster than polynomial  decay of correlations and finally the integrability of return time is used to get the result for the flow.

Using the main result of \cite{S06}, by a similar construction,   if the flow also satisfies the following property

\begin{description}
\item[3)] the map $T$ has  derivative bounded by a power law near $c$: there is a $\beta >0 $ s.t.  $(x-c)^{\beta}T'(x)$ is bounded in a neighborhood of $c$
\end{description}

  we prove (in Section \ref{sec:recurrpoli}, Corollary 7.4) the following estimation for the return time 

\textbf{Theorem C (quantitative recurrence) }\emph{If the flow satisfies conditions 1),2), 3) above, then for a.e. }$x$ \emph{it holds }
\begin{equation}
\underset{r\rightarrow 0}{\lim \sup }\frac{\log \tau _{r}^{\prime }(x)}{%
-\log r}=\overline{d}_{\mu _{X}}(x)-1,\quad \underset{r\rightarrow 0}{\lim \inf }%
\frac{\log \tau _{r}^{\prime }(x)}{-\log r}=\underline{d}_{\mu _{X}}(x)-1.
\end{equation}

In the Appendix II we give an auxiliary result, using
a theorem by Steinberger \cite{S00} showing that the local dimension
is defined a.e. for the Geometric Lorenz system.


\section{Geometric Lorenz model}
\label{sec:Lorenzmodel}

In this section we will introduce and motivate the so-called Geometric
Lorenz system. This is the main example where our results will be
applied. Indeed we will see that assumption 1.a),...,1.d) and 2) of the introduction are verified for this model.
The results in this section are however not strictly necessary
for the proofs of our main theorems. The reader
familiar with the construction of such models can skip it
and start at Section \ref{sec:SBRfluxo}.

In 1963 the meteorologist Edward Lorenz published in the Journal of
Atmospheric Sciences (\cite{Lo63}) an example of a parametrized
$2$-degree polynomial system of differential equations
\begin{align}
   \label{e-Lorenz-system}
\dot x &= a(y - x)&
\quad
&a = 10 \nonumber
\\
\dot y &= rx -y -xz&
\quad
&r =28
\\
\dot z &= xy - bz&
\quad
&  b = 8/3 \nonumber
\end{align}
as a very simplified model for thermal fluid convection,
motivated by an attempt to understand the foundations of
weather forecast.

Numerical simulations performed by Lorenz for an open neighborhood
of the chosen parameters suggested that almost all points in phase space
tend to a {\em chaotic attractor}.

An {\em attractor} is a bounded region in phase-space, invariant
under time evolution, such that the forward trajectories of most
(positive probability) or, even, all nearby points converge to.
And what makes an attractor {\em chaotic} is the fact that trajectories
converging to the attractor are {\em sensitive with respect to
initial data}: trajectories of two any nearby points get apart under
time evolution.

\begin{figure}[h]
  \centering
  \includegraphics[width=9cm]{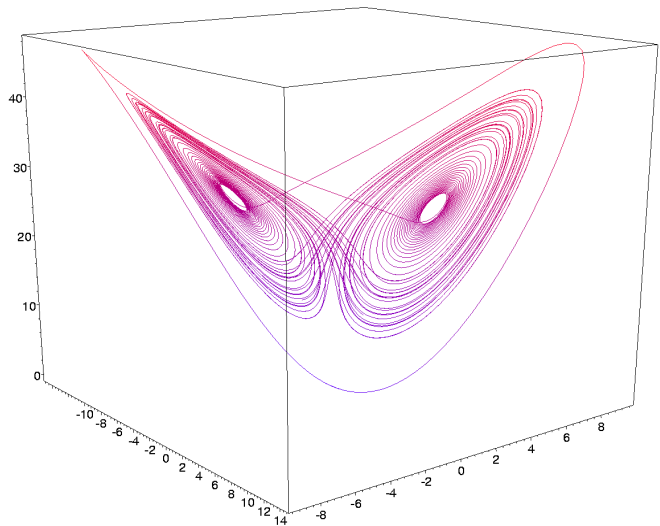}
  \caption{\label{fig-Lorenz3D} Lorenz chaotic attractor}
\end{figure}
Lorenz's equations proved to be very resistant to
rigorous mathematical analysis, and also presented
serious difficulties to rigorous numerical study.
As an example, the above stated existence of a chaotic attractor for the original Lorenz
system where not proved until the year 2000, when Warwick Tucker did it with
a computer aided proof (see~\cite{Tu99,Tu2}).

In order to construct a class of flows having properties which are
very similar to the Lorenz system  and are easier to be studied,
Afraimovich, Bykov and Shil'nikov \cite{ABS77}, and Guckenheimer,
Williams \cite{GW79}, independently constructed the so-called
\emph{geometric Lorenz models} for the behavior observed by Lorenz.
These models are flows in $3$-dimensions for which one can
rigorously prove the existence of a chaotic attractor that contains
an equilibrium point of the flow, which is an accumulation point of
typical regular solutions. Recall that $\gamma$ is a regular
solution for the flow $X^t$ if $X^t(x)\neq x$ for  all $x\in \gamma$ and $t>0$.
 The accumulation of regular orbits
near an equilibrium prevents such sets from being hyperbolic \cite{PM82}.
Furthermore, this attractor is robust (in the $C^1$ topology): it can not be
destroyed by any small perturbation of the original flow.

We point out that the robustness of this example provides an open set of flows
which are not Morse-Smale, nor hyperbolic, and also non-structurally stable (see \cite{PM82,BDV05} ).



\subsection{Construction of the geometric model: near the equilibrium}
\label{sec:near-singul}
\label{sec:constr-geometr-model}

The results of the paper will be given for a class of three dimensional flows which will be defined axiomatically.
To show that these axioms are verified in the geometric Lorenz models we give a detailed introduction to this model.

We first analyze the dynamics in a neighborhood of the
singularity at the origin, and then we complete the flow, imitating the butterfly shape of the original Lorenz flow (see Figure \ref{fig-Lorenz3D} and compare with Figure \ref{L3D}).

In the original Lorenz system the origin $p=0=(0,0,0)$ is an equilibrium of saddle type for the vector field defined by equations (\ref{e-Lorenz-system}) with real eigenvalues $\lambda_i$, $i\leq 3$ satisfying
\begin{equation}
\label{eigenvalues}
0<\frac{\lambda_1}{2}\leq-\lambda_3<\lambda_1<-\lambda_2
\end{equation}
(in the classical Lorenz system  $\lambda_1\approx 11.83$ , $\lambda_2\approx -22.83$, $\lambda_3=-8/3$).

If   certain nonresonance conditions are satisfied (see \cite{St58})
this  vector field is smoothly linearizable in a neighborhood of the
origin. To construct a model which is similar to the original Lorenz
one we start with a linear system $(\dot x, \dot y, \dot
z)=(\lambda_1 x,\lambda_2 y, \lambda_3 z)$, with $\lambda_i$, $1\leq
i\leq 3$ satisfying relation (\ref{eigenvalues}). This vector field
will be considered in the cube $[-1,1]^3$ containing the origin.

 For this linear flow, the trajectories are given by
\begin{align}\label{eq:LinearLorenz}
X^t(x_0,y_0,z_0)=
(x_0e^{\lambda_1t}, y_0e^{\lambda_2t}, z_0e^{\lambda_3t}),
\end{align}
where  $(x_0, y_0, z_0)\in\RR^3$ is an
arbitrary initial point near $p=(0,0,0)$.

Consider $\Sigma=\big\{ (x,y,1) : |x|\le
{\frac{1}{2}},\quad |y|\le{\frac{1}{2}}\big\}$ and
\begin{align*}
\Sigma^-&=\big\{ (x,y,1)\in \Sigma : x<0 \big\},&
\qquad
\Sigma^+&=\big\{ (x,y,1)\in \Sigma : x>0 \big\}\quad\text{and}
\\
\Sigma^*&=\Sigma^-\cup \Sigma^+=\Sigma\setminus\Gamma,& \quad\text{where}\quad
\Gamma&=\big\{(x,y,1)\in \Sigma : x=0 \big\}.
\end{align*}
 $\Sigma$  is a  transverse section to the linear flow and
every trajectory  crosses $\Sigma$ in the direction of
the negative $z$ axis.

Consider also
$\tilde{\Sigma}=\{ (x,y,z) : |x|=1\}=\tilde{\Sigma}^-\cup {\tilde{\Sigma}}^+$ with
${\tilde{\Sigma}}^{\pm}=\{ (x,y,z): x=\pm 1\}$.
For each $(x_0,y_0,1)\in \Sigma^*$ the time $t$
such that $X^{t}(x_0,y_0,1)\in\tilde{\Sigma}$ is given by
\begin{equation}\label{tempo}
t(x_0)=-\frac{1}{\lambda_1}\log{|x_0|}
\end{equation}
 which depends on $x_0\in \Sigma^*$ only and is such that $t(x_0)\to+\infty$
when $x_0\to0$.

 Hence, using (\ref{tempo}), we get
(where $\sgn(x)=x/|x|$ for $x\neq0$ )
\begin{align*}
X^{t(x_0)}(x_0,y_0,1)=
\big( \sgn(x_0),  y_0e^{\lambda_2\cdot t(x_0)},
e^{\lambda_3\cdot t(x_0)}\big)
=
\big( \sgn(x_0),
y_0|x_0|^{-\frac{\lambda_2}{\lambda_1}},
|x_0|^{-\frac{\lambda_3}{\lambda_1}}\big).
\end{align*}
Since $0<\frac{\lambda_1}{2}<-\lambda_3<\lambda_1<-\lambda_2$, we  have
$\frac{1}{2}<\alpha=-\frac{\lambda_3}{\lambda_1} <1
<\beta=-\frac{\lambda_2}{\lambda_1}$.

Consider
 $L:\Sigma^*\to\ {\tilde{\Sigma}}^{\pm}$  defined by
\begin{equation}\label{L}
L(x,y,1)=\big(\sgn(x),
y|x|^\beta,|x|^\alpha\big).
\end{equation}
\begin{figure}[h]
\begin{center}
\includegraphics[width=7cm]{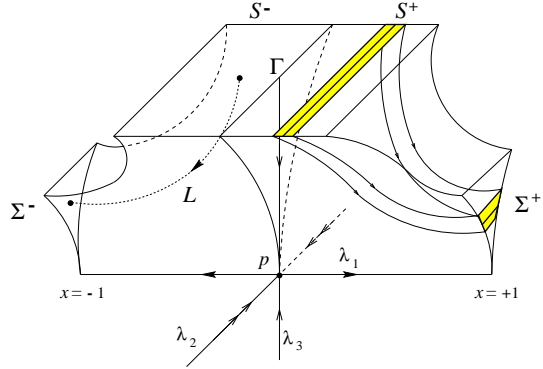}
\end{center}
\caption{\label{L3Dcusp}Behavior near the origin.}
\end{figure}
It is easy to see that $L(\Sigma^\pm)$ has the shape of a cusp
triangle without the vertex $(\pm 1,0,0)$.
In fact the vertex $(\pm 1,0,0)$ are
cusp points at the boundary of each of these sets.
The fact that $0<\alpha<1<\beta$ together with equation (\ref{L}) imply
that $L(\Sigma^\pm)$ are uniformly compressed in the $y$-direction.

 Clearly each segment $\Sigma^*\cap\{x=x_0\}$ is
taken by $L$ to another segment $\tilde{\Sigma}^\pm\cap\{z=z_0\}$ as
sketched in Figure~\ref{L3Dcusp}.
\begin{figure}[h]
\begin{center}
\includegraphics[width=6cm]{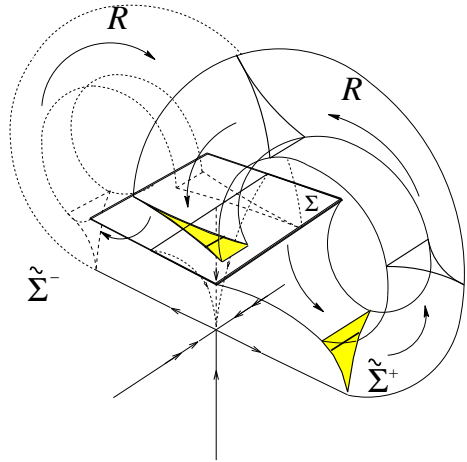}
\end{center}
\caption{\label{L3D}$T_\pm\circ R_\pm$ takes $\tilde\Sigma^\pm$ to $\Sigma$.}
\end{figure}

\subsection{The random turns around the origin}
\label{rotacao}

To imitate the random turns of a regular orbit around the origin and obtain
a butterfly shape for our flow, as it is in the original Lorenz flow depicted
at Figure \ref{fig-Lorenz3D},  we proceed as follows.

Recall that the equilibrium $p$ at the origin is hyperbolic and so its
stable $W^s(p)$ and unstable $W^u(p)$ manifolds are well defined, \cite{PM82}.
Observe that $W^u(p)$ has dimension one and so, it has two branches,
$W^{u,\pm}(p)$, and $W^u(p)=W^{u,+}(p)\cup\{p\}\cup W^{u,-}(p)$.

The sets $L(\Sigma^\pm)$ should return to the cross section $\Sigma$
through a  flow described by a suitable composition of a rotation $R_\pm$,
an expansion $E_{\pm\theta}$
and a translation $T_\pm$.

The rotation $R_\pm$ has axis parallel to the
$y$-direction. More
precisely is such that $(x,y,z)\in \tilde\Sigma^\pm$, then
\begin{align}
 \label{derivadadeR}
R_\pm(x,y,z)=
\left(
\begin{array}{cccc}
      0 & 0 & \pm 1       \\
 0 & 1 & 0\\
\pm 1 & 0 & 0
\end{array}
\right).
\end{align}
The expansion occurs only along the $x$-direction, so, the matrix of
$E_{\theta}$ is given by

\begin{align}
 \label{dilatacao}
E_{\pm\theta}(x,y,z)=
\left(
\begin{array}{cccc}
    \theta   & 0 & 0       \\
 0 & 1 & 0\\
0 & 0 & 1
\end{array}
\right)
\end{align}
with $\theta \cdot ({\frac{1}{2}}^\alpha )<1 $ and $\theta\cdot \alpha\cdot 2^{1-\alpha} > 1$. The first condition is to ensure that the image of the resulting map is contained in $\Sigma$, the second condition makes a certain one dimensional induced map to be piecewise expanding. This point will be discussed below.

$T_\pm:\RR^3\to \RR^3$ is chosen such that the unstable direction starting
from the origin is sent to the boundary of $\Sigma $ and the image
of both $\tilde\Sigma^\pm$ are disjoint.
 These transformations $R_\pm , E_{\pm\theta}, T_\pm $ take line
segments $\tilde\Sigma^\pm\cap\{z=z_0\}$ into line segments
$\Sigma\cap\{x=x_1\}$ as sketched in Figure~\ref{L3D}, and so does
the composition $T_\pm\circ E_{\pm\theta}\circ R_\pm$.

This composition of linear maps describes a vector field in a region outside
$[-1,1]^3$ in the sense that one can use the above matrices to define a
vector field $V$ such that the time one map of the associated flow
realizes $T_\pm\circ E_{\pm\theta}\circ R_\pm$ as a map $L(\Sigma^\pm) \to \Sigma$. This will not be explicit here, since the choice of the vector field  is not really important for our purposes (provided the return time is integrable).

The above construction allow to describe for each $t\in\mathbb{R} $
the orbit $X^t(x)$  of each point $x \in \Sigma$: the orbit will
start following the linear field until  $\tilde\Sigma^\pm$ and then
it will follow $V$ coming back to $\Sigma$ and so on. Let us denote with
${\cal  B}=\{ X^t(x),x\in \Sigma, t\in \mathbb{R}^+\} $ the set where this flow acts.
The geometric Lorenz flow is then the couple $({\cal B}, X^t )$ defined in this way.

The Poincar\'e first return map will be hence defined by
$F:\Sigma^*\to \Sigma$ as
\begin{equation}
 \label{F}
F(x,y)=\left\{
\begin{array}{ccc}
 T_+\circ E_{+\theta}\circ R_+\circ {L}(x,y,1) & \mbox{for }\, x>0\\
T_-\circ E_{-\theta}\circ R_-\circ {L}(x,y,1) &  \mbox{for }\, x < 0
\end{array}
\right.
\end{equation}

The combined effects of $T_\pm\circ R_\pm$ and ${L}$ on lines implies that
the foliation $\cF^s$ of $\Sigma$ given by the lines $\Sigma\cap\{x=x_0\}$ is
invariant under the return map. In another words, we have
\vspace{0.2cm}

$(\star)$ \/{\em for any given
leaf $\gamma$ of $\cF^s$, its image $F(\gamma)$ is
contained in a leaf of $\cF^s$}.
\vspace{0.2cm}

\subsection{An expression for the first return map and its differential}
\label{exprassaodeF}

Combining equations (\ref{L}) with the effect of the rotation composed with
the expansion and the translation,
we obtain that $F$ must have the form
\begin{equation}\label{Fgrande}
F(x,y)=\big(f_{Lo}(x),g_{Lo}(x,y)\big)
\end{equation}
where
 $f_{Lo}:I\setminus\{0\}\to I$ and
$g_{Lo}:(I\setminus\{0\})\times I\to I$ are given by
\begin{equation}
  \label{fLo}
f_{Lo}(x)=
\left\{
\begin{array}{cccc}
 f_1(x^\alpha) &  x < 0       \\
f_0(x^\alpha) &  x > 0
\end{array}
\right.
\quad \mbox{with $f_i =(-1)^{i}\theta\cdot x+b_i, i\in\{0,1\}$, and }
\end{equation}
\begin{equation}
  \label{gLo}
g_{Lo}(x,y)=
\left\{
\begin{array}{cccc}
g_1(x^\alpha,y\cdot x^\beta) &  x < 0       \\
g_0(x^\alpha,y\cdot x^\beta) &  x > 0,
\end{array}
\right.
\,\,
\end{equation}
where  $g_1|I^-\times I\to I $ and $g_0|I^+\times I\to I$ are suitable affine maps. Here $I^-=(-1/2,0)$, $I^+=(0,1/2)$.
\begin{figure}[htbp]
\begin{minipage}{5cm}
\centering{\includegraphics[width=5cm]{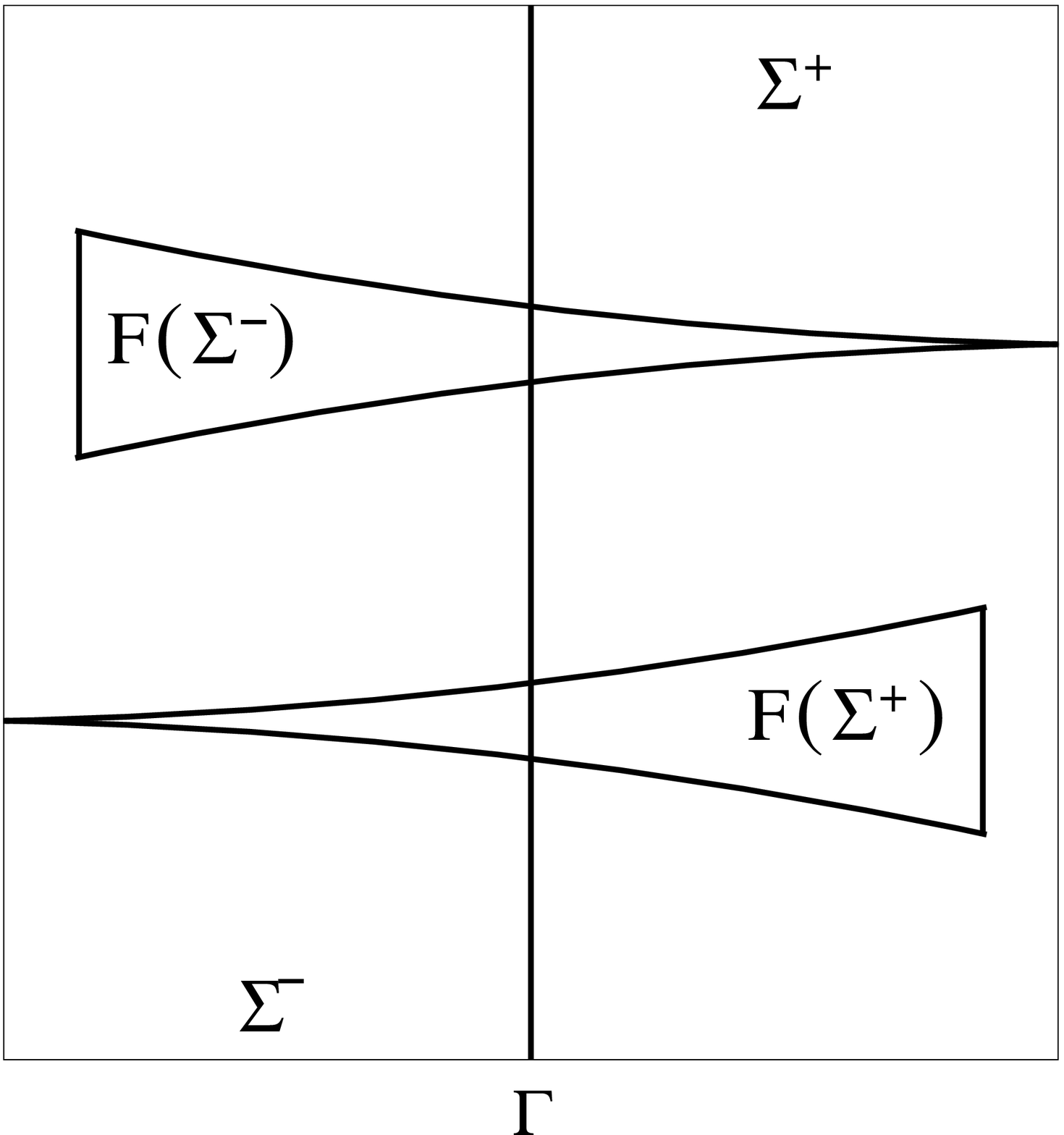}}
\caption{\label{L2D}{$F(\Sigma^*)$.}}
\end{minipage}
\hfill
\begin{minipage}{5cm}
\centering{\includegraphics[width=5cm]{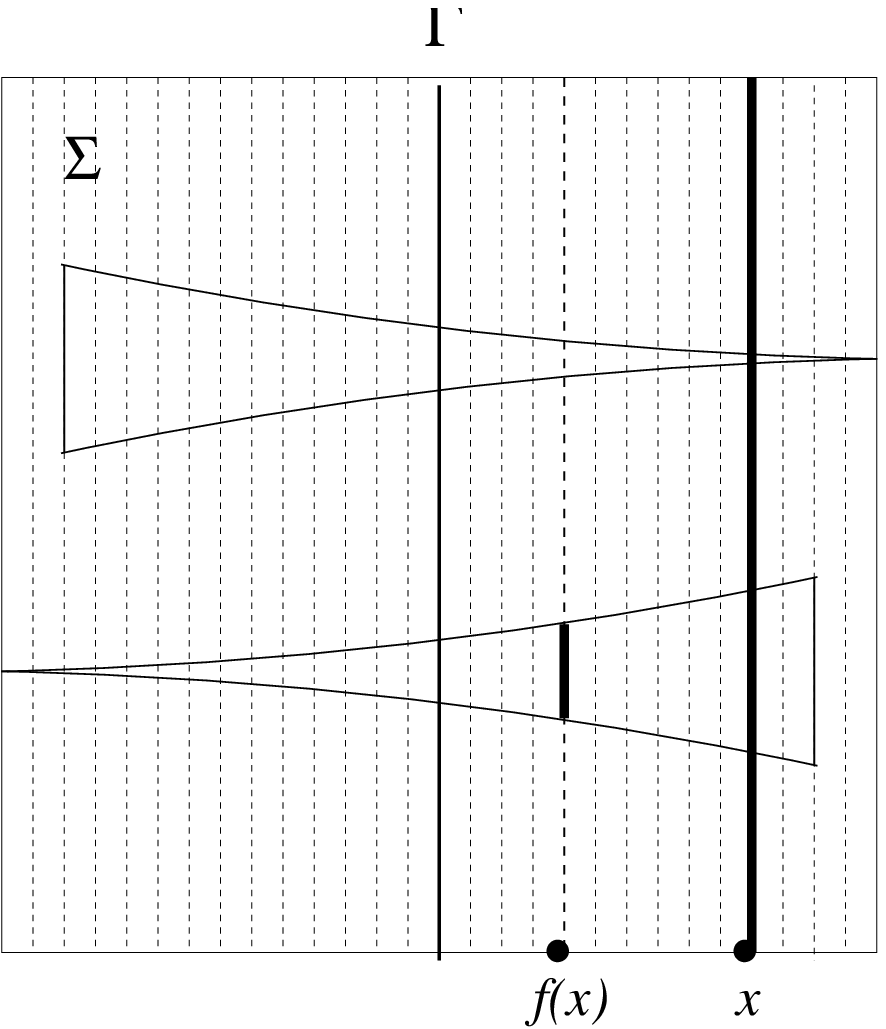}}
\caption{\label{folhea}{Projection on $I$.}}
\end{minipage}
\end{figure}

Now, to find an expression for $DF$ we proceed as follows.
Recall $F=T_\pm \circ E_{\pm \theta} \circ R_{\pm} \circ L$,
$L$ is as in  (\ref{L}),  $DR_\pm$ is as in (\ref{derivadadeR}).
Given $q=(x,y) \in \Sigma^*$ with $x>0$, we have

\begin{align*}
DL(x,y,1)=
\left(
\begin{array}{cccc}
\beta \cdot y\cdot x^{\beta-1} &  x^\beta       \\
\alpha \cdot x^{\alpha-1} &  0  \\
\end{array}
\right).
\end{align*}

Restricting the rotation and the other linear maps to $\tilde\Sigma^\pm$ and composing the resulting matrices we  get

\begin{align}
 \label{eq:derivadaF}
DF(x,y)=  \left(
\begin{array}{cccc}
    \theta\cdot\alpha\cdot x^{(\alpha-1)} &  0
  \\
  \beta\cdot
  yx^{(\beta-\alpha)}
  & x^\beta
\end{array}
\right).
\end{align}

The expression for $DF$ at $q=(x,y)$ with $x< 0$ is similar.


\subsection{Properties of the  map $g_{Lo}$}
\label{asegundacoordenada}

Observe that by construction $g_{Lo}$ in equation (\ref{F}) is piecewise $ C^2$.
Moreover, equation (\ref{eq:derivadaF}) implies the following bounds on
its partial derivatives :

\begin{enumerate}
 \item[(a)] For all $(x,y)\in\Sigma^*, x> 0$, we have
${\partial_y} g_{Lo}(x,y)=  x^\beta$. As $\beta>1$, $|x|\leq 1/2$,
there is $0<\lambda<1$ such that
\begin{equation}
 \label{gy}
|{\partial_y} g_{Lo}| < \lambda.
\end{equation}
The same bound works for $x < 0$.

\item[(b)] For all $(x,y)\in\Sigma^*, x \neq 0 $, we have
${\partial_x} g_{Lo}(x,y)=\beta\cdot x^{\beta-\alpha}$.
As $\beta-\alpha > 0$  and $|x|\leq 1/2$,  we get
\begin{equation}
 \label{gx}
|\partial_x g_{Lo}| < \infty.
\end{equation}
\end{enumerate}
Item (a) above implies that
the map $F=(f_{Lo},g_{Lo})$ is uniformly contracting on the leaves of the foliation $\cF^s$:
there is $ C >0$ such that, if $\gamma$ is a leaf of $\cF^s$ and  $x,y\in\gamma$ then
$\dist\big(F^n(x),F^n(y)\big)\leq {\lambda^ n}\cdot C\cdot  \dist(x,y)$
where $\lambda$ can be chosen as the one given by equation (\ref{gy}).

\begin{figure}[htpb]
\centerline{\includegraphics[width=4cm]{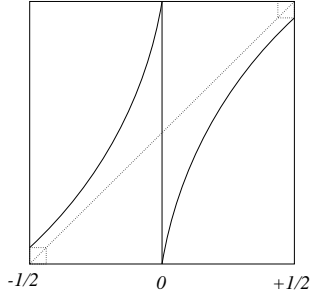}}
\caption{\label{L1D}{The Lorenz map $f_{Lo}$.}}
\end{figure}

\subsection{Properties of the one-dimensional map $f_{Lo}$}
\label{sec:propert-one-dimens}

Now let us outline the main properties of $f_{Lo}$.
  We recall that we chosen $\theta$ such that $\theta\cdot \alpha\cdot 2^{1-\alpha} > 1$. 

The following properties are easily implied from the
construction of $X^t$:
\begin{enumerate}
\item[(f1)] By equation (\ref{fLo}) and the way $T_\pm$ is defined,
$f_{Lo}$ is discontinuous at $x=0$. The lateral limits
  $f_{Lo}(0^\pm)$ do exist,  $f_{Lo}(0^\pm)=\pm\frac12$,

\item[(f2)] $f_{Lo}$ is $C^2$ on $I\setminus\{0\}$. By the choice of $\theta$  it holds $f'_{Lo}(1/2) >1$. By the convexity properties of $f_{Lo}$ we then obtain that
\begin{equation}
 \label{derivadamaiorqueum}
f'_{Lo}(x) > 1 \quad \quad \mbox{for all}\quad \quad x \in I\setminus\{0\}.
\end{equation}

\item[(f3)] The limits of $f_{Lo}'$ at $x=0$ are
  $\lim_{x\to 0 }f_{Lo}'(x)=+\infty$.
\end{enumerate}

We obtain that $f_{Lo}$ is a piecewise expanding map.
Moreover $f_{Lo}$ has a dense orbit, which in its turn implies
that the closure of the maximal invariant set by $f_{Lo}$
is the whole interval $I$, see \cite[Lemma 2.11]{APbook}.

Now recall that the {\em variation} $\vari \phi$ of a function $\phi:[0,1]\to \RR$ is defined by
$$
\vari \phi= \sup \sum_{i=1}^{n}|\phi(x_{i-1})-\phi(x_i)|
$$
where the supremum is taken over all finite partitions
$0=x_0<x_1<\cdots < x_n=1$, $n\geq 1$, of $[0,1]$. The variation
$\vari_J \phi=\vari(\phi|J)$ of $\phi$ over an arbitrary interval
$J\subset [0,1]$ is defined by a similar expression, with the
supremum taken over all the $x_0, x_1, \cdots , x_n \in J$, with
$\inf J \leq x_0< x_1 < \cdots < x_n \leq \sup J$.
One says that $\phi$ has {\em bounded variation}, or $\phi$ is BV for short,
 if $\vari  \phi <  \infty$.

The one dimensional map has the following property, which is important to obtain  the existence of an SRB invariant measure and its statistical properties.

\begin{lemma}
\label{boundedvariation}
 Let $X^t$ a $C^2$ geometric Lorenz flow as before and $f_{Lo}$ be the one-dimensional map
associated to $X^t$. Then $\frac{1}{f_{Lo}'}$ is BV.
\end{lemma}

\begin{proof}
Each branch of $f_{Lo}$ is the composition of an affine map with $x^\alpha$ then it is a convex function.
Hence, the derivative $f'_{Lo}$ is monotonic on each branch, implying that $(f'_{lo})^{-1}$ is also monotonic.
 On the other hand, $(f'_{Lo})^{-1}$ is bounded because $f'_{lo}>1$.
Thus $(f'_{Lo})^{-1}$ is monotonic and bounded and hence is BV.
\end{proof}

We have seen that $f_{Lo}$ is a topologically transitive piecewise expanding map
with $\frac{1}{f_{Lo}'}$  BV. The statistical properties of such maps are well
known. Next we state a result about it, which will be used later:

\begin{proposition}
 \label{prop:densidadeBVdef}(\cite{Vi97}, Prop.3.8)
The one-dimensional  $f_{Lo}$ admits a unique invariant probability
$\mu_{f_{Lo}}$ which is absolutely continuous with respect to
Lebesgue measure $m$, it is ergodic and so a SRB measure for the
map. Moreover $d\mu_{f_{Lo}}/dm$ is  a BV function and in particular
it is bounded. Furthermore $f_{Lo}$ has exponential decay of
correlations for $L^1$ and BV observables and any a.c.i.m. converges
exponentially fast to the invariant measure:
 there are constants $C>0$ and $\lambda >0$, depending on the
system such that for each $n$ and observables $ f , g$:%
\begin{equation*}
\left|\int ~g(F^{n}(x))f(x)dm-\int g(x)d\mu \int f(x)dm \right|\leq C\cdot \Vert g\Vert
_{L_{1}}\cdot \Vert f\Vert _{BV}\cdot e^{-\lambda n}.
\end{equation*}
\end{proposition}


Summarizing, for what it was said above, the study of the $3$-flow can be reduced to the study of a bi-dimensional map $F$. Moreover, the dynamics of this map
can be further reduced to a one-dimensional map, $f_{Lo}$ called one dimensional {\em{ Lorenz map}}.
Figure~\ref{L1D} shows the graph of this one-dimensional transformation, and Figure~\ref{L2D} sketches $F(\Sigma^*)$.


\section{A physical measure for a Lorenz like flow}
\label{sec:SBRfluxo}

In this section, following \cite{Vi97} we construct a physical measure for a flow $X^t$ which satisfies the assumptions 1a),...,1d),2) in the introduction. As noticed in the previous section, these assumptions 1a),...,1d) are satisfied by the Geometric Lorenz system.

Properties 1a),...,1d) implies that the flow Poincar\'e map has an invariant foliation and the one dimensional induced map $T$ is piecewise expanding.
Piecewise expanding maps (see Proposition \ref{prop:densidadeBVdef}) admits a unique invariant probability  measure $\mu_{T}$ which is absolutely continuous with respect to Lebesgue measure $m$.

From $\mu_{T}$ we may construct  a SRB measure $\mu_F$, for the first return map  $F$ through the following general procedure (\cite{Bo75,Vi97}).
Since $\mu_{T}$ is defined on the interval $I$ which can be
identified to the space of leaves of the contracting foliation $\cF^s$, we may
also think of it as a measure on the $\sigma$-algebra of Borel subsets of
$\Sigma$ which are union of entire leaves of $\cF^s$. Using the fact that
$F$ is uniformly contracting on leaves of $\cF^s$  we conclude that the sequence
$$
F^{*n}(\mu_{T}), \quad n\geq 1,
$$
of push-forwards of $\mu_{T}$ under $F$ is weak*-Cauchy:
given any continuous $\psi: \Sigma\to \RR$
$$
\int \psi d(F^{n*}\mu_{T})=\int (\psi\circ F^n) d{\mu_{T}}, \quad n \geq 1,
$$
is a Cauchy  sequence in $\RR$, see \cite[pp.173]{Vi97}.
Define $\mu_F$ to be the weak*-limit of this sequence, that is,
$$
\int \psi d\mu_F=\lim \int \psi d(F^{*n}\mu)
$$
for each continuous $\psi$.
Then $\mu_F$ is invariant under $F$, and it is an ergodic physical measure for $F$.
The last statement follows from the fact that $\mu_{T}$ is an ergodic physical measure for $T$, together with the fact that asymptotic time-averages of
continuous functions $\psi:\Sigma \to \RR$ are constant on the leaves of $\cF^s$.

Given any point $x$ whose orbit sooner or later will cross $\Sigma$ we denote with $t(x)$ the first strictly positive time such that $X^{t(x)}(x)\in\Sigma$ (the {\em return time} of $x$ to $\Sigma$). Coherently with the Geometric Lorenz system, we will denote by $\Sigma^* $ the (full measure, by the assumption 1 in the introduction) subset of $\Sigma $ where $t$ is defined.

Now we show how to construct an physical invariant measure for the flow, {\em when the return time is integrable}:
\begin{equation}
\int_{\Sigma^*} t d\mu_F < \infty.
\end{equation}

Denote by $\sim$ the equivalence relation on $\Sigma\times \RR$ given by
$(w,t(w))\sim(F(w),0).$

Let $N=(\Sigma^*\times\RR)/\sim$ and $\nu=\pi_*(\mu_F\times dt)$, where $\pi:\Sigma^*\times\RR\to  N$
is the quotient map and $dt$ is a Lebesgue measure in $\RR$.
Equation (\ref{eq:integraldotempo-limitado}) gives that $\nu$  is a finite measure.
Let $\phi:N\to \RR^3$ be defined by $\phi(w,t)=X^t(w)$.
Let $\mu_X=\phi_*\nu$.
The measure $\mu_X$ is a physical for the flow $X^t$:
$$
\frac{1}{T}\int_0^T \psi(X^t(w))dt\to \int \psi d\mu_X \quad \mbox{as}
\quad T\to \infty
$$
for every continuous function $\psi:\RR^3\to \RR$, and Lebesgue almost every point
$w\in \phi(N).$

We end the subsection remarking that the Geometric Lorenz flow has integrable return time, hence the above construction for the invariant measure can be applied to it. As before denote by $t:\Sigma\setminus\Gamma \to (0,\infty)$ the return time to $\Sigma$. Then, recalling Equation (\ref{tempo}) there are $K,C>0$ such that

\begin{equation*}\label{eq:tempo}
-K^{-1}\log (d(x,\Gamma))-{C}\leq t(x)\leq -K\log (d(x,\Gamma))+C .
\end{equation*}

Combining this with the definition of $\mu_F$ and the remark made above that $d\mu_{f_Lo}/dm$ is a bounded function,
we conclude that
\begin{proposition} The return time is integrable
\begin{equation}
 \label{eq:integraldotempo-limitado}
t_0=\int t d\mu_F < \infty.
\end{equation}
\end{proposition}

\subsection{Local dimension.}
\label{sec:invariantmeasure}

Let us recall the  definition of local dimension and fix some notations for what follows.

Let $(M,d)$ be a metric space and assume that $\mu$ is a
Borel probability measure on $M$.
Given $x\in M$, let $B_r(x)=\{y\in M; d(x,y)\leq r\}$ be the ball centered at $x$ with radius $r$.
The {\em local dimension} of $\mu$ at $x\in M$ is defined by
$$
d_\mu(x)=\lim_{r\to 0}\frac{\log\mu(B_r(x))}{\log r}
$$
if this limit exists.
In this case $\mu(B_r(x)) \sim r^{d_\mu(x)}$.

This notion characterizes the local geometric structure of an invariant measure with respect to the metric in the phase space of the system, see \cite{Yo82} and \cite{Pe97}.

We can always define the {\em upper} and the {\em lower} local dimension at $x$ as
$$
{\overline d}_\mu(x)=\lim \sup_{r\to 0}\frac{\log\mu(B_r(x))}{\log r}\,,\quad\quad
{\underline d}_\mu(x)=\lim \inf_{r\to 0}\frac{\log\mu(B_r(x))}{\log r}.
$$
If $d^+(x)=d^-(x)=d$ almost everywhere the system is called {\em exact dimensional}.
In this case many properties  of dimension of a measure coincide.
In particular, $d$ is equal to the infimum Hausdorff dimension of full measure sets:
$d=\inf\{\dim_H Z; \mu(Z)=1\}.$
This happens in a large class of systems, for example, in $C^{2}$ diffeomorphisms having non
zero Lyapunov exponents almost everywhere, \cite{Pe97}.

\subsection{Relation between local dimension for $F$ and for $X^t$}
\label{sec:local-hitting}

Let us establish a relation between  $d_{\mu_{F}}$ and $d_{\mu_X}$ which will be used in the following.

\begin{proposition}\label{dimensao}
Let $x\in \mathbb{R}^3$ and $\pi (x)$ be the projection on $\Sigma $ given by $\pi (x)=y $ if $x$ is on the orbit of $y\in \Sigma$ and the orbit from $y$ to $x$ does not cross $\Sigma$ (if $x \in \Sigma$ then $\pi (x)=x$).
For all regular points $x\in \mathbb{R}^3 $ 
 \begin{equation}
  \label{relacao}
{{\overline d}_{\mu_X}}(x)={{\overline d}_{\mu_F}}(\pi(x)) + 1,\ \ {{\underline d}_{\mu_X}}^{-}(x)={{\underline d}_{\mu_F}}^{-}(\pi(x)) + 1.
\end{equation}
\end{proposition}
\begin{proof}
First observe that for product measures as
${\mu_X}=\mu_F\times dt$, where $dt$ is the Lebesgue measure at the line,
the formula is trivially verified.
But, by construction 
$\mu_X=\phi_\ast(d\mu_F\times dt)$, where $\phi:\RR^3\to \RR^3$
is a local bi-Lipschitz map at each regular point.
Since the local dimension is invariant by local bi-Lipschitz maps, it follows  the required equation (\ref{relacao}).
\end{proof}


\section{Decay of correlations for two dimensional Lorenz maps}

\label{2p1} In this section we estimate the decay of correlations for a
class of Lorenz like maps containing the first return map of the geometric
Lorenz system described above. Inspired by a remark of R. S. McKay (see 
\cite{Mc}, p. 8), this will be done by estimating the speed of approaching
of iterates of suitable measures (corresponding to Lipschitz observables) to
the invariant measure. For this purpose we will consider the space of
measures on $\Sigma $ as a metric space, endowed with the
Wasserstein-Kantorovich distance, whose basic properties we are going to
describe.

\textbf{Notations. }Let us introduce some notations: we will consider the $%
\sup $ distance on $\Sigma =[-\frac{1}{2},\frac{1}{2}]^{2}$, so that the
diameter, $diam(\Sigma )=1$. This choice is not essential, but will avoid
the presence of many multiplicative constants in the following making
notations cleaner.

As before, the square $\Sigma $ will be foliate by stable, vertical leaves.
We will denote the leaf with $x$ coordinate by $\gamma _{x}$ or, with a
small abuse of notation, when no confusion is possible we will denote both
the leaf and its coordinate with $\gamma $.

Let $f\mu $ be the measure $\mu _{1}$ such that $d\mu _{1}=fd\mu $.
Moreover, let us sometime for short denote the integral by $\mu (f)=\int
fd\mu .$ Let $\mu $ a measure on $\Sigma $. In the following, such measures
on $\Sigma $ will be often disintegrated in the following way: for each
Borel set $A$%
\begin{equation}
\mu (A)=\int_{\gamma \in I}\mu _{\gamma }(A\cap \gamma )d\mu _{x}
\label{dis}
\end{equation}%
with $\mu _{\gamma }$ being probability measures on the leaves $\gamma $ and 
$\mu _{x}$ is the marginal on the $x$ axis which will be an absolutely
continuous probability measure. We will also denote by $\phi _{x}$ its
density.

Let us consider the projection $\pi _{y}$ on the $y$ coordinate. Let us
denote the "restriction" of $\mu $ on the leaf $\gamma $ by 
\begin{equation*}
\mu |_{\gamma }=\pi _{y}^{\ast }(\phi _{x}(\gamma )\mu _{\gamma }).
\end{equation*}%
This is a measure on $I$ and it is not normalized. We remark that $\mu
|_{\gamma }(I)=\phi _{x}(\gamma )$. If $Y$ is a metric space, we denote by $%
PM(Y)$ the set of Borel probability measures on $Y$. Let us finally denote
by $L(g)$ be the best Lipschitz constant of $g:$ $L(g)=\sup_{x,y}\frac{%
|g(x)-g(y)|}{|x-y|}$ and set $\Vert g\Vert _{lip}=\Vert g\Vert _{\infty
}+L(g).$

\subsection{The Wasserstein-Kantorovich distance\label{sec:W-Kdistance}}

Let us consider a bounded metric space $Y$ and let us consider the following
notion of distance between measures: given two probability measures $\mu
_{1} $ and $\mu _{2}$ on $Y$

\begin{equation*}
W_{1}(\mu _{1},\mu _{2})=\underset{g\in 1lip(Y)}{\sup }(|\int_{Y}g~d\mu
_{1}-\int_{Y}g~d\mu _{2}|)
\end{equation*}%
where $1lip(Y)$ is the space of $1$-Lipschitz functions on $Y.$ We remark
that adding a constant to the test function $g$ does not change the above
difference $\int g~d\mu _{1}-\int g~d\mu _{2}.$ The above defined $W_{1}$
has moreover the following basic properties.

\begin{proposition}
\cite[Prop 7.1.5]{AGS}\label{prop:ambros} The following properties hold

\begin{enumerate}
\item $W_{1}$ is a distance and if $Y$ is separable and complete, then $%
PM(Y) $ with this distance is a separable and complete metric space.

\item  A sequence is convergent for the $W_{1}$ metrics if
and only if it is convergent for the weak topology.
\end{enumerate}
\end{proposition}

\begin{remark}
\label{conv}(distance and convex combinations) If $a+b=1,a\geq 0,b\geq 0$
then%
\begin{equation}
W_{1}(a\mu _{1}+b\mu _{2},a\mu _{3}+b\mu _{4})\leq a\cdot W_{1}(\mu _{1},\mu
_{3})+b\cdot W_{1}(\mu _{2},\mu _{4}).
\end{equation}%
Indeed%
\begin{equation*}
W_{1}(a\mu _{1}+b\mu _{2},a\mu _{3}+b\mu _{4})=\underset{g\in 1lip(Y)}{\sup }%
(|\int g~d(a\cdot \mu _{1}+b\cdot \mu _{2})-\int g~d(a\cdot \mu _{3}+b\cdot
\mu _{4})|)=
\end{equation*}%
\begin{equation*}
=\underset{g\in 1lip(Y)}{\sup }(|a\cdot \int g~d\mu _{1}+b\cdot \int g~d\mu
_{2}-a\cdot \int g~d\mu _{3}-b\cdot \int g~d\mu _{4}|)
\end{equation*}%
\begin{equation*}
\leq \underset{g\in 1lip(Y)}{\sup }(|a\int g~d\mu _{1}-a\cdot \int g~d\mu
_{3}|+|b\cdot \int g~d\mu _{2}-b\cdot \int g~d\mu _{4}|)=
\end{equation*}%
\begin{equation*}
\underset{g\in 1lip(Y)}{\sup }(a\cdot |\int g~d\mu _{1}-\int g~d\mu
_{3}|+b\cdot |\int g~d\mu _{2}-\int g~d\mu _{4}|)\leq a\cdot W_{1}(\mu
_{1},\mu _{3})+b\cdot W_{1}(\mu _{2},\mu _{4}).
\end{equation*}%
We also remark that the same kind of estimation can be done if the convex
combination has more than 2 summands.
\end{remark}

\begin{remark}
If $g$ is $\ell $-Lipschitz and $\mu _{1},\mu _{2}$ are probability measures
then 
\begin{equation*}
|\int_{Y}g~d\mu _{1}-\int_{Y}g~d\mu _{2}|\leq \ell \cdot W_{1}(\mu _{1},\mu
_{2}).
\end{equation*}
\end{remark}

\subsection{Wassertein distance and decay of correlations over Lipschitz
observables}

We give some general facts on the relation between $W_{1}$ distance and
decay of correlations.

Let $(Y,F,\mu )$ be a dynamical system on a metric space with invariant
probability\ measure $\mu $. The transfer operator associated to $F$ will be
indicated with $F^{\ast }$.

\begin{proposition}[decay as function of distance]
\label{prop:1}Let  $g\in lip(Y)$ and $f\in L^1(Y,\mu )$, $f\geq 0$. Let $\mu _{1}$ be a probability measure which is absolutely continuous with respect to $\mu ,$ and \thinspace $d\mu _{1}=\frac {f(x)}{||f||_{L^1}}d\mu $. Then
\begin{equation}\label{decwass}
|\int g(F^{n}(x))~f(x)d\mu -\int f(x)d\mu\int g(x)d\mu |\leq L(g)\cdot {||f||_{L^1}} \cdot W_{1}((F^{\ast
})^{n}(\mu _{1}),\mu ).
\end{equation}
\end{proposition}

\begin{proof}
Dividing by $L(g)$ we can suppose $g\in 1lip(Y)$. As $\int g(F(x))\frac {f(x)}{||f||_{L^1}}d\mu
=\int g(x)d(F^{\ast }(\mu _{1}))$ then the decay of correlations between $f$
and $g$ can be estimated in function of the distance between $(F^{\ast
})^{n}(\mu _{1})$ and $\mu $ as$:$%
\begin{equation*}
L(g) ||f||_{L^1}|\int g(F^{n}(x))~\frac {f(x)}{||f||_{L^1}}d\mu -\int g(x)d\mu |=L(g) ||f||_{L^1}|\int g(x)d(F^{\ast n}(\mu
_{1}))-\int g(x)d\mu |
\end{equation*}

\begin{equation*}
\leq L(g) ||f||_{L^1} \underset{g\in 1lip(Y)}{\sup }(|\int g~d(F^{\ast }{}^{n}(\mu
_{1}))-\int g~d\mu |)= L(g) ||f||_{L^1} W_{1}((F^{\ast })^{n}(\mu _{1}),\mu ).
\end{equation*}
\end{proof}

Conversely,

\begin{proposition}[distance as function of decay]
\label{prop:2}If for each $f\in L^{1}(\mu ),$ $f\geq 0$ and $g\in lip(Y)$ it
holds 
\begin{equation*}
|\int g(F^{n}(x))~f(x)d\mu -\int f(x)d\mu \int g(x)d\mu |\leq C\cdot
L(g)\cdot \Vert f\Vert _{L^{1}}\cdot \Phi (n)
\end{equation*}%
then taking $d\mu _{1}=\frac{f(x)}{||f||_{L^{1}}}d\mu $ it holds%
\begin{equation*}
W_{1}((F^{\ast })^{n}(\mu _{1}),\mu )\leq C\cdot \Phi (n).
\end{equation*}
\end{proposition}

\begin{proof}
Consider $g\in 1lip$. Hence%
\begin{eqnarray*}
\frac{C\cdot L(g)\Vert f\Vert _{L^{1}}\cdot \Phi (n)}{||f||_{L^{1}}} &\geq &%
\frac{|\int g(F^{n}(x))~f(x)d\mu -\int f(x)d\mu \int g(x)d\mu |}{%
||f||_{L^{1}}}= \\
&=&|\int g(x)d(F^{\ast n}(\mu _{1}))-\int g(x)d\mu |
\end{eqnarray*}%
since this hold for each $g$ hence $W_{1}(F^{\ast }{}^{n}(\mu _{1}),\mu
)\leq C\cdot \Phi (n).$
\end{proof}

\subsection{Disintegration and Wasserstein distance}

We will consider maps having an invariant foliation, as we have seen in the
Lorenz map. The invariant measure will then be disintegrated as in Equation (%
\ref{dis}) into a family of measures $\mu _{\gamma }$ on almost each stable
leaf $\gamma $ and an absolutely continuous measure $\mu _{x}$ on the
unstable direction.

If $\mu ^{1}$ and $\mu ^{2}$ are two disintegrated measures as above, their $%
W_{1}$ distance can be estimated in function of some distance between their
respective marginals on the $x$ axis and measures on the leaves:

\begin{proposition}
\label{prod}Let $\mu ^{1}$, $\mu ^{2}$ be measures on $\Sigma $ as above,
such that for each Borel set $A$

\begin{itemize}
\item $\mu ^{1}(A)=\int_{\gamma \in I}\mu _{\gamma }^{1}(A\cap \gamma )d\mu
_{x}^{1}$

\item $\mu ^{2}(A)=\int_{\gamma \in I}\mu _{\gamma }^{2}(A\cap \gamma )d\mu
_{x}^{2}$
\end{itemize}

with $\mu _{x}^{i}$ absolutely continuous with respect to the Lebesgue
measure, moreover let us suppose

\begin{enumerate}
\item for almost each vertical leaf $\gamma $, $W_{1}(\mu _{\gamma }^{1},\mu
_{\gamma }^{2})\leq \epsilon $ and

\item $\sup_{\Vert h\Vert _{\infty }\leq 1}|\int hd\mu _{x}^{1}-\int hd\mu
_{x}^{2}|\leq \delta $
\end{enumerate}

then $W_{1}(\mu ^{1},\mu ^{2})\leq \epsilon +\delta .$
\end{proposition}

\begin{proof}
Considering the $W_{1}$ distance and disintegrating $\mu ^{1}$ and $\mu ^{2}$%
:

\begin{equation}
W_{1}(\mu ^{1},\mu ^{2})\leq \sup_{g\in 1lip}|\mu ^{1}(g)-\mu ^{2}(g)|=
\end{equation}%
\begin{equation*}
=\sup_{g\in 1lip}|\int_{\gamma \in I}\int_{\gamma }g(\ast )d\mu _{\gamma
}^{1}d\mu _{x}^{1}-\int_{\gamma \in I}\int_{\gamma }g(\ast )d\mu _{\gamma
}^{2}d\mu _{x}^{2}|.
\end{equation*}%
Adding and subtracting $\int \int_{\gamma }g(\ast )d\mu _{\gamma }^{2}d\mu
_{x}^{1}$ the last expression is equivalent to 
\begin{eqnarray*}
&&\sup_{g\in 1lip}|\int_{I}\int_{\gamma }g(\ast )d\mu _{\gamma }^{1}d\mu
_{x}^{1}-\int_{I}\int_{\gamma }g(\ast )d\mu _{\gamma }^{2}d\mu _{x}^{1}+ \\
&&+\int_{I}\int_{\gamma }g(\ast )d\mu _{\gamma }^{2}d\mu
_{x}^{1}-\int_{I}\int_{\gamma }g(\ast )d\mu _{\gamma }^{2}d\mu _{x}^{2}|.
\end{eqnarray*}%
This becomes%
\begin{equation*}
\sup_{g\in 1lip}|\int_{I}(\int_{\gamma }g(\ast )d\mu _{\gamma }^{1}-g(\ast
)d\mu _{\gamma }^{2})d\mu _{x}^{1}+\int_{I}\int_{\gamma }g(\ast )d\mu
_{\gamma }^{2}d\mu _{x}^{1}-\int_{I}\int_{\gamma }g(\ast )d\mu _{\gamma
}^{2}d\mu _{x}^{2}|\leq
\end{equation*}%
\begin{equation*}
\leq \sup_{g\in 1lip}|\int_{I}\epsilon d\mu _{x}^{1}+\int_{I}\int_{\gamma
}g(\ast )d\mu _{\gamma }^{2}d\mu _{x}^{1}-\int_{I}\int_{\gamma }g(\ast )d\mu
_{\gamma }^{2}d\mu _{x}^{2}|\leq
\end{equation*}%
\begin{equation}
\leq \epsilon +|\int_{I}\int_{\gamma }g(\ast )d\mu _{\gamma }^{2}d\mu
_{x}^{1}-\int_{I}\int_{\gamma }g(\ast )d\mu _{\gamma }^{2}d\mu _{x}^{2}|
\end{equation}%
Since $g\in 1lip$ and $diam(\Sigma )=1$ (on the square we consider the sup distance), then by adding a constant to $g$
(which does not change $\int gd\mu _{\gamma }^{1}-\int gd\mu _{\gamma }^{2}$
) we can suppose without loss of generality that $g\leq 1$ and then for
almost each $\gamma $ it holds $h(\gamma )=|\int_{\gamma }g(\ast )d\mu
_{\gamma }^{2}|\leq 1$. Hence, by assumption (2) the statement is proved.
\end{proof}

\subsection{Exponential decay of correlations. \label{sec:primeiroretorno}}

Now we are ready to prove that a Lorenz like two dimensional map $F$ has
exponential decay of correlations with respect to its SRB measure $\mu $. We
recall (see Proposition \ref{prop:densidadeBVdef} ) that for a piecewise
expanding map of the interval $T$, there are constants $C>0$ and $\lambda >0$%
, depending on the system such that, if $g$ and $f$ are respectively $L^{1}$
and BV (bounded variation) observables on $I$ for each $n$ it holds:%
\begin{equation}
|\int ~g(T^{n}(x))f(x)dm-\int g(x)d\mu \int f(x)dm|\leq C\cdot \Vert g\Vert
_{L_{1}}\cdot \Vert f\Vert _{BV}\cdot e^{-\lambda n}  \label{l1bv}
\end{equation}%
(recall that $m$ is the Lebesgue measure above). This will be used in the proof of the following theorem
\begin{theorem}
\label{resuno}Let $F:\Sigma \rightarrow \Sigma $ a Borel function such that $%
F(x,y)=(T(x),G(x,y))$. Let $\mu $ be an invariant measure for $F$ with
marginal $\mu _{x}$ on the $x$-axis (which is invariant for $T:I\rightarrow
I $ ). Let us suppose that

\begin{enumerate}
\item $(T,\mu _{x})$ satisfies the above equation \ref{l1bv} and $T^{-1}(x)$
is finite for each $x\in I$.

\item $F$ is a contraction on each vertical leaf: $G$ is $\lambda $-Lipschitz in $y$ with $\lambda <1$ for each $x$ .

\item $\mu $ is regular enough that for each $\ell $-Lipschitz function $%
f:\Sigma \rightarrow \mathbb{R}$ the projection $\pi _{x}^{\ast }(f\mu )$
has bounded variation density $\overline{f}$ \footnote{%
which can also be expressed as $\overline{f}(x)=\int f(x,y)~d\mu |_{\gamma
_{x}}.$}, with%
\begin{equation}
var(\overline{f})\leq K\ell  \label{fbar}
\end{equation}
where $K$ is not depending on $f$.
\end{enumerate}

Then $(F,\mu )$ has exponential decay of correlation (with respect to
Lipschitz  and $L^1$ observables as in Equation \ref{decwass} ).
\end{theorem}

We already saw that the first two points in the above proposition are
satisfied by the first return map of the Geometric Lorenz system. In the
Appendix I we will prove that also the above item 3 is satisfied by the family of systems described in the introduction, containing the Geometric Lorenz one.

We point out that this is the hard part of the proof that Lorenz like maps
have exponential decay of correlations and this will be done by a sort of
Lasota-Yorke inequality. Putting together all the necessary assumptions, this
prove Theorem A in the introduction.

Before the proof of Theorem \ref{resuno} we  make the following remark
which is a simple but important fact implied by the uniform contraction on
stable leaves

\begin{remark}
\label{remark}Under the above assumptions, let us consider a leaf $\gamma $
and two probability measures $\mu $, $\nu $ on it. Then 
\begin{equation*}
W_{1}(F^{\ast }(\mu ),F^{\ast }(\nu ))\leq \lambda W_{1}(\mu ,\nu ).
\end{equation*}

\begin{proof}
This is because the map is uniformly contracting on each leaf. If $g$ is $1$%
-Lipschitz on $F(\gamma )$ then $g(F(\ast ))$ is $\lambda $-Lipschitz on $%
\gamma $. This implies that 
\begin{equation*}
|\int_{F(\gamma )}g~d(F^{\ast }\mu )-\int_{F(\gamma )}g~d(F^{\ast }\nu
)|=|\int_{\gamma }g\circ F~d\mu -\int_{\gamma }g\circ F~d\nu |\leq \lambda
W_{1}(\mu ,\nu )
\end{equation*}%
finishing the proof.
\end{proof}
\end{remark}

\begin{proof}
(of Theorem \ref{resuno}) Let us consider $\nu =f\mu $ with $f\geq 0$ being $%
\ell -$Lipschitz and $\int fd\mu =1$ (remark that this implies $\ell \geq 1 $). The strategy is to use Proposition \ref%
{prod} and find exponentially decreasing bounds for $\epsilon $ and $\delta $
so that we can estimate the Wasserstein distance between $\mu $ and iterates
of $f\mu $ and then apply Proposition \ref{prop:1} to deduce decay of
correlations from the distance. Let us consider the leaf $\gamma _{x}$ with coordinate $x$.
The density $\overline{f}$ , by item 3 has bounded variation and 
$|| \overline{f}||_{BV}\leq K\ell +1\leq (K+1)\ell $. Let $\nu _{x}=\overline{f}m$ the measure on the $x$-axis with density $\overline{f}$ (as before $m$ is the Lebesgue measure). Let us consider the base map $T$. Let $g\in
L^{1}([-\frac{1}{2},\frac{1}{2}])$. Since $|\int g~d(T^{\ast n}(\nu
_{x}))-\int g~d\mu _{x}|=|\int g(T^{n}(x))\overline{f}(x)dm-\int g(x)d\mu
_{x}|$, by equation (\ref{l1bv})%
\begin{equation*}
|\int gd(T^{\ast n}(\nu _{x}))-\int gd\mu _{x}|\leq \Vert g\Vert
_{L_{1}}\cdot \Vert \overline{f}\Vert _{BV}\cdot C\cdot e^{-\lambda n},
\end{equation*}%
implying that $\sup_{\Vert g\Vert _{\infty }\leq 1}|\int gdT^{\ast n}(\nu
_{x})-\int gd\mu _{x}|\leq \Vert \overline{f}\Vert _{BV}\cdot C\cdot
e^{-\lambda n}\leq (K+1)\ell C\cdot e^{-\lambda n}$ and hence we see that item
(2) at Proposition \ref{prod} is satisfied with an exponential bound
depending on the Lipschitz constant $\ell $ of $f$.

Let us consider $\nu ^{n}=F^{\ast n}\nu $ again. Since, as said before the
map $F$ sends vertical leaves into vertical ones then there is a family of
probability measures $\nu _{\gamma }^{n}$ on vertical leaves such that%
\begin{equation*}
(F^{\ast n}\nu )(g)=\int_{\gamma \in I}\int_{\gamma }g(\ast )d\nu _{\gamma
}^{n}d((T^{\ast n}(\nu _{x}))).
\end{equation*}%
To satisfy item (1) at Proposition \ref{prod} and hence conclude the
statement we only have to prove that there are $C_{2},\lambda _{2}$ s.t.%
\begin{equation*}
\forall \gamma \,\,\,~W_{1}(\nu _{\gamma }^{n},\mu _{\gamma })\leq
C_{2}\cdot e^{-\lambda _{2}n}
\end{equation*}
this is because of uniform contraction on stable leaves.

Indeed, by remark \ref{remark}, if $\nu _{\gamma }$ and $\rho _{\gamma }$
are the two probability measures on the leaf $\gamma $ then the measures $%
F^{\ast }(\nu _{\gamma }),F^{\ast }(\rho _{\gamma })$ on the contracting
leaf $F(\gamma )$ are such that%
\begin{equation*}
W_{1}(F^{\ast }(\nu _{\gamma }),F^{\ast }(\rho _{\gamma }))\leq \lambda
\cdot W_{1}(\nu _{\gamma },\rho _{\gamma }).
\end{equation*}%
Now let us consider $F^{-1}(\gamma )=\gamma _{1}\cup \gamma _{2}...\cup
\gamma _{k}$ \ and apply the above inequality to estimate the distance of
iterates of the measure on the leaves. For simplicity let us show the case
where the pre-image of a leaf consists of two leaves as it happen in the
Geometric Lorenz system, the case where the pre-image consists of more leaves
is analogous: let hence $F^{-1}(\gamma )=\gamma _{1}\cup \gamma _{2}$, after
one iteration of $F^{\ast }$ on  $\nu $ and $\mu $ the "new"
measures $\nu _{\gamma }^{1}=(F^{\ast }(\nu ))_{\gamma }$ and $\mu _{\gamma
} $ (which is equal to $(F^{\ast }(\mu ))_{\gamma }$ because $\mu $ is
invariant) on the leaf $\gamma $ will be a convex combination of the images
of the "old" measures on $\gamma _{1}$ and $\gamma _{2}$%
\begin{equation*}
\nu _{\gamma }^{1}=a\cdot F^{\ast }(\nu {}_{\gamma _{1}})+b\cdot F^{\ast
}(\nu {}_{\gamma _{2}}),
\end{equation*}%
\begin{equation}
\mu _{\gamma }=a\cdot F^{\ast }(\mu {}_{\gamma _{1}})+b\cdot F^{\ast }(\mu
{}_{\gamma _{2}})
\end{equation}%
with $a+b=1,a,b\geq 0$ \ (the second equality is again because $\mu $ is
invariant). By the triangle inequality (remark \ref{conv})

\begin{equation*}
W_{1}(\nu _{\gamma }^{1},\mu _{\gamma })\leq a\cdot W_{1}(F^{\ast }(\nu
{}_{\gamma _{1}}),F^{\ast }(\mu {}_{\gamma _{1}}))+b\cdot W_{1}(F^{\ast
}(\nu {}_{\gamma _{2}}),F^{\ast }(\mu {}_{\gamma _{2}}))
\end{equation*}

and by remark \ref{remark} 
\begin{equation*}
W_{1}(\nu _{\gamma }^{1},\mu _{\gamma })\leq \lambda (a\cdot W_{1}(\nu
_{\gamma _{1}},\mu _{\gamma _{1}})+b\cdot W_{1}(\nu _{\gamma _{2}},\mu
_{\gamma _{2}}))
\end{equation*}%
hence 
\begin{equation*}
W_{1}(\nu _{\gamma }^{1},\mu _{\gamma })\leq \lambda \sup_{\gamma
}(W_{1}(\nu {}_{\gamma },\mu _{\gamma })).
\end{equation*}%
The same can be done in the case when the pre-image $F^{-1}(\gamma )=\gamma
_{1}$ is only one leaf or more than two, hence by induction $W_{1}(\nu
_{\gamma }^{n},\mu _{\gamma })<\lambda ^{n}$, and the exponential bound on
the distance of iterates on the leaves (item 1 of Proposition \ref{prod}) is
provided.
\end{proof}


 \section{Hitting time: flow and section}\label{sec:hittingtime}

We now consider again a Lorenz like flow, with integrable return
time, i.e. a flow $X^t$ having a transversal section $\Sigma $ whose first return map satisfies the assumptions of Theorem \ref{resuno}  and the
return time is integrable, as before. As before $F:\Sigma\setminus
\Gamma \to \Sigma$ is the first return map associated.

Let $x, x_{0}\in \RR^3$ and
\begin{equation*}
\tau _{r}^{X^t}(x,x_{0})=\inf \{t\geq 0 | X^{t}(x)\in B_r(x_{0})\}
\end{equation*}
be the time needed for the $X$-orbit of a
point $x$ to enter for the {\em first time} in a ball
$B_{r}(x_{0})$. The number $\tau_{r}^{X^t}(x,x_{0})$
is the {\em{hitting time associated to}} the flow $X^t$ and $B_r(x_0)$.

If $x, x_{0}\in \Sigma$ and $B_{r}^{\Sigma}(x_{0})=B_r(x_0)\cap\Sigma$,
we define
$$\tau _{r}^{\Sigma }(x,x_{0})=\min\{n\in\NN^+; F^n(x)\in B_{r}^{\Sigma}(x_{0})\}:$$
the {\em{hitting time associated to}} the discrete
system $F$.

Given any $x$ we recall that we denoted with $t(x)$ the first strictly positive time, such that $X^{t(x)}(x)\in\Sigma$ (the {\em return time} of $x$ to $\Sigma$).
A relation between ${\tau _{r}}^X(x,x_{0})$ and $\tau _{r}^{\Sigma }(x,x_{0})$ is
given by

\begin{proposition}\label{hittingsecaoelivre} Under the above assumptions, if $\int_{\Sigma }t(x)~d\mu _{F}<\infty$, then, there is $K\geq 0$ and a set $A\subset \Sigma$ having full $\mu_{F} $ measure such that for each $x_0\in\Sigma$, $x\in A$  
\begin{equation}\label{eq:relationhitting}
 c(x,r)\cdot \tau_{Kr}^{\Sigma }(x,x_{0})\cdot \int_{\Sigma }t(x)~d\mu _{F}\leq \tau_{r}^{X^t}(x,x_{0})\leq c(x,r)\cdot
\tau_{r}^{\Sigma }(x,x_{0})\cdot \int_{\Sigma }t(x)~d\mu _{F}
\end{equation}%
with $c(x,r)\rightarrow 1$ as $r\rightarrow 0$.
\end{proposition}

\begin{proof}
Let us assume that $x,x_{0}$ $\in $ $\Sigma $, $x\neq x_0 $ and $r\leq d(x,x_0 )$. Since the flow cannot hit the section near $x_0$ without  entering in a small ball of the space centered at $x_0$ before, then
$\tau_{r}^{\Sigma }(x,x_{0})$ and $\tau_{r}^{X^t}(x,x_{0})$ are related by
\begin{equation}\label{sum1}
\tau_{r}^{X^t}(x,x_{0})\leq \sum_{i=0}^{\tau _{r}^{\Sigma }(x,x_{0})}t(F^{i}(x)).
\end{equation}
Moreover, since the section is supposed to be transversal to the flow, there is a $K$ such that 
\begin{equation}\label{sum2}
\tau_{r}^{X^t}(x,x_{0})\geq \left[\sum_{i=0}^{\tau _{Kr}^{\Sigma }(x,x_{0})}t(F^{i}(x)) \right].
\end{equation}
The last inequality follows by the fact that if the flow at some time crosses the ball centered at $x_0$ then after a time $e(r)$ it will cross the section at a distance less than $Kr$, where $K$ depends on the angle between the flow and the section (when $r$ is small approximate  locally the flow by a constant one).

The above sums are Birkhoff sums of the observable $t$ on
the $F$-orbit of $x$ and $\mu _{F}$ is ergodic.
Then there is a full measure set $A\subset \Sigma$ (and $x_0 \notin A$)  such that
\begin{equation*}
\frac{1}{n}\sum_{i=0}^{n}t(F^{i}(x))\longrightarrow \int_{\Sigma }t(x)~d\mu _{F},\quad \mbox{as}\quad n\to\infty
\end{equation*}%
for $x\in A$. Hence
\begin{equation*}
\frac{1}{\tau _{r}^{\Sigma }(x,x_{0})}\sum_{i=0}^{\tau _{r}^{\Sigma }(x,x_{0})}t(F^{i}(x))\longrightarrow \int_{\Sigma }t(x)~d\mu _{F},\quad \mbox{as}\quad n\to\infty
\end{equation*}
for $x\in A$. Thus we get that for each $x\in A$
\begin{equation}\label{eq:finale}
\sum_{i=0}^{\tau _{r}^{\Sigma }(x,x_{0})}t(F^{i}(x))=c(x,r)\cdot \tau _{r}^{\Sigma
}(x,x_{0})\cdot \int_{\Sigma }t(x)~d\mu _{F}
\end{equation}%
with $c(x,r)\rightarrow 1$ as $r\rightarrow 0$.
Combining  Equations (\ref{sum1},\ref{sum2}) and (\ref{eq:finale}) we get (\ref{eq:relationhitting}).
\end{proof}

Let $\pi$ be the projection on $\Sigma $ defined in Proposition \ref{dimensao}.
The above statement implies the following

\begin{proposition}\label{c:relacaodimensao}
There is a full measure set $B\subset {\mathbb R}^3$ (for the flow invariant measure) such that if $x_0\in  {\mathbb R}^3$ is regular and $x\in B$ it holds (provided the limits exist)
\begin{equation}\label{eq:tempos}
\lim_{r\rightarrow 0}\frac{\log \tau_{r}^{X^t}(x,x_{0})}{-\log r}%
=\lim_{r\rightarrow 0}\frac{\log \tau _{r}^{\Sigma }(\pi(x),\pi(x_{0}))}{-\log r}.
\end{equation}
\end{proposition}

\begin{proof}
The above Proposition implies that if $x_{0},x\in \Sigma $  and $x\in A$ then \begin{equation}
\lim_{r\rightarrow 0}\frac{\log \tau_{r}^{X^t}(x,x_{0})}{-\log r}
=\lim_{r\rightarrow 0}\frac{\log \tau _{r}^{\Sigma }(x,x_{0})}{-\log r}.
\end{equation}
  If $
x_{0}\in \mathbb{R}^{3}$ is a regular point, the flow $X$ induces a
bilipschitz homeomophism from a neighborhood of $\pi (x_{0})\in \Sigma $ to a neighborhood of  $x_{0}$.

Hence there is $K\geq 1$ such that $$\tau _{K^{-1}r}^{X}(x,\pi
(x_{0}))+Const\leq \tau _{r}^{X}(x,x_{0})\leq \tau _{Kr}^{X}(x,\pi
(x_{0}))+Const$$ where $Const$ represents the time which is needed to go from 
$\pi (x_{0})$ to $x_{0}$ by the flow.  This is also true for each $x\in B=\pi ^{-1}(A)$.
Extracting logarithms and taking the
limits we get the required result.
\end{proof}

We recall that  (see Section \ref{sec:SBRfluxo}) the assumption $\int_{\Sigma }t(x)~d\mu_{F}<\infty$ is verified for  the geometric Lorenz flow. Hence these results applies for this example.


 \section{A logarithm law for the hitting time}
\label{sec:loglaw}


In this section we give the main result for the behavior of the hitting time on Lorenz like flows. First let us recall a result on discrete time systems.

Let $(Y,T,\mu )$ be a measure preserving (discrete time) dynamical system.
We say that $(X,T,\mu)$ has super-polynomial decay of correlations with respect to Lipschitz observables if
$$
\left|\int\varphi\circ T^n\psi \cdot d\mu -
\int \varphi \cdot d\mu \cdot \int \psi\cdot d\mu \right|
\leq \|\varphi\|\cdot \|\psi\|\cdot \theta_n,
$$
where $\lim_n\theta_n\cdot n^p=0$ for all $p> 0$ and $\|\cdot\|$ is the Lipschitz norm.

In \cite{Ga07} the following fact is proved for discrete time systems:
\begin{theorem}\label{qui} Let $(Y,T,\mu)$ a measure preserving transformation having superpolynomial decay of correlations as above.
For each $x_{0}\in Y$ such that $d_{\mu }(x_{0})$ is defined
$$
\lim_{r\rightarrow 0}\frac{\log \tau _{r}(x,x_{0})}{-\log r}%
=d_{\mu }(x_{0})
$$
for $\mu$-almost each $x\in Y$.
\end{theorem}

Applying this  to the 2-dimensional system $(\Sigma, F, \mu_{F})$ (which satisfies the assumptions of Theorem \ref{resuno} since  ans hence has exponential decay of correlations). We  conclude the following

\begin{corollary}\label{th:stefano}
Let $F:\Sigma \to \Sigma $ be a map with an invariant measure $\mu
_{F}$ satisfying the assumptions of Theorem \ref{resuno}. For each
$x_0\in \Sigma$ such that $d_{\mu _{F}}(x_{0})$ exists then
$$
\lim_{r\rightarrow 0}\frac{\log \tau _{r}^{\Sigma }(x,x_{0})}{-\log r}%
=d_{\mu _{F}}(x_{0}).
$$
for $\mu_F$-almost $x\in \Sigma$.
\end{corollary}

Now, if we consider a flow having such a map as its Poincar\'e section and integrable return time, we can construct as in Section \ref{sec:SBRfluxo} an SRB invariant measure $\mu_X$ for the flow.
 By Proposition \ref{c:relacaodimensao}, Corollary \ref{th:stefano} and Proposition \ref{dimensao} we can estimate the hitting time to balls for the flow by the corresponding estimation for the Poincar\'e map and we get our main result, which corresponds to
Theorem B in the introduction (where a set of sufficient assumptions on the map are listed):

\begin{theorem}\label{main1}
If $X^t$ is a Lorenz like flow, that is a flow having a transversal
section, with a Poincar\'e map satisfying the assumptions of
proposition \ref{resuno} and integrable return time, then for each regular
$x_{0}\in \RR^3$ such that $d_{\mu_X}(x_{0})$ exists, it holds
\begin{equation*}
\lim_{r\rightarrow 0}\frac{\log \tau_{r}^{X^t}(x,x_{0})}{-\log r}=d_{\mu_X}(x_{0})-1
\end{equation*}
for $\mu_X$-almost each $x\in \RR^3.$
\end{theorem}

\section{ Quantitative recurrence for Lorenz like systems}
\label{sec:recurrpoli}
\label{sec:recurrLorenz}

We now recall a general result proved by Saussol in \cite{S06} about
quantitative recurrence in order to apply it to a Lorenz like flow.
The result shows that the power law behavior of the return time in small balls
can be estimated by function of the local dimension if the system has
fast enough decay of correlations.

Given a set $A$, we denote the boundary of $A$ as  $\partial A$.

\begin{theorem}\cite[Thm 4, Lemma 13]{S06}.
\label{th:Saussol}
Let $(Y,T,\mu )$ be a measure preserving dynamical system, where $Y$ is a Borel subset of some euclidean space.
Assume that the entropy $h_{\mu }(T)>0$
and $T$ is such that there exists a partition ${\cA}$ (modulo $\mu $)
into open sets  such that for each $A\in \cA$ the map $T|_A$ is Lipschitz
with constant $L_{T}(A)$.
Furthermore, suppose that

\begin{enumerate}
\item the set  ${\cS(\cA)}=\cup\{\partial A\in \cA\}$ is such that there are constants $c>0$ and $a>0$ so that
\begin{equation*}
\mu \left(\{x\in X:\dist(x,\cS(\cA) )<\epsilon \}\right)< c\cdot \epsilon^{a}.
\end{equation*}
\item the average Lipschitz exponent
\begin{equation*}
\sum_{A\in \cA} \mu (A)\log^{+}L_{T}(A)
\end{equation*}
is finite,
\item the decay of correlation of $T$ is super-polynomial with respect to Lipschitz observables.
\end{enumerate}
Then
\begin{equation}
\liminf_{r\to 0}\frac{\log \tau_r(x,x)}{-\log r}={d}^-_{\mu }(x)~,\quad \mbox{and}\quad \,\,
\limsup_{r\to 0}\frac{\log \tau_r(x,x)}{-\log r} ={d}^+_{\mu }(x)~a.e.
\end{equation}
\end{theorem}

Let us first show that the above theorem can be applied to the Geometric Lorenz system. 

\begin{lemma}
The first return map $(F,\Sigma ,\mu_F )$ of the Geometric Lorenz system (described in Section \ref{sec:Lorenzmodel}) satisfies the hypothesis of Theorem \ref{th:Saussol} above.
\end{lemma}

\begin{proof}
Since we have proved that the system $(F,\Sigma,\mu_F)$ is
exponentially mixing, item  (3) at Theorem \ref{th:Saussol} is satisfied.

The partition $\cA=\{A_{i}\},$ with
\begin{equation}
A_{i}=[\left(\frac{1}{i+2},\frac{1}{i+1}\right)\cup \left(\frac{-1}{i+2},\frac{-1}{i+1}\right) ]\times \mathring{I},\ i\in {\mathbb N}^+
\end{equation}%
where $\mathring{I}$ denotes the interior of $I$, satisfies  (1) and (2) at Theorem \ref{th:Saussol}.
Here we note that $F$ is not globally Lipschitz, but from Eq. (\ref{eq:derivadaF}) we get $L_{T}(A_{i})\leq K\cdot i^\beta$, with $\beta > 1$ and $K> 0$.

Moreover, the fact that $\mu_F$ has a bounded density marginal (the density will be denoted by $f_0$ as before) on the $x$ direction  implies that the measure of the sets $A_i$ can be estimated by
$$\mu (A_{i})\leq \frac{4\cdot\sup(f_{0})}{i^{2}}.$$

Thus,
$$
\sum_{A\in \cS(\cA)}\log^{+}L_{F}(A)\cdot \mu (A)=\sum_{A\in \cS(\cA)}\log^+ (K\cdot i^\beta)\cdot \frac{4\cdot\sup(f_{0})}{i^{2}}< \infty.
$$
This finishes the proof.
\end{proof}

In the same way, replacing Equation \ref{eq:derivadaF} with assumption 3) in the introduction it can be proved that the above theorem applies to Lorenz like flows:

\begin{lemma}\label{l:FSaussol}
If the system $(F,\Sigma,\mu_F)$ is the first return map of a flow satisfying  assumptions 1.a),...1.d),2),3) of the introduction, then it
 satisfies the hypothesis of Theorem \ref{th:Saussol}.
\end{lemma}

Applying Theorem \ref{th:Saussol} to such system, then we get

\begin{corollary}
 \label{co:FSuassol}
For the system $(F,\Sigma,\mu_F)$ it holds
$$
\liminf_{r\to 0}\frac{\log \tau_r^\Sigma(x,x)}{-\log r}={\underline d}_{\mu_F},\quad\quad\quad
\limsup_{r\to 0}\frac{\log \tau_r^\Sigma(x,x)}{-\log r}={\overline d}_{\mu_F},
\quad \mu_F-a. e.\,.
$$
\end{corollary}

Finally, remarking that regular points have full measure, with the same arguments as in Proposition \ref{c:relacaodimensao} by Proposition \ref{dimensao}, we get

\begin{corollary}
 \label{co:XSaussol}
For the Geometric Lorenz flow and for  Lorenz like flows
 as above it holds
$$
\liminf_{r\to 0}\frac{\log \tau'_r(x)}{-\log r}={\underline d}_{\mu_X}-1,\quad\quad\quad
\limsup_{r\to 0}\frac{\log \tau'_r(x)}{-\log r}={\overline d}_{\mu_X}-1,
\quad \mu_X-a. e.\,.
$$
where $\tau'$ is the recurrence time for the flow, as defined in the introduction.
\end{corollary}
This is the content of Theorem C in the introduction.


\section{Appendix I: about regularity of the measure $\protect\mu _{F}$}

In this section we are going to prove that the SRB measure of a Lorenz like
map satisfies item 3 of Theorem \ref{resuno}. We remark that this is a kind
of regularity assumption for the measure $\mu _{F}$ (a certain projection is
BV). The proof is done in several steps and it will be completed at the end
of the section. The statement we are going to prove is:

\begin{lemma}
\label{resdue}Let $F(x,y)=(T(x),G(x,y))$ be a  map preserving the
vertical foliation, such that:

\begin{enumerate}
\item There is $c\in I$ and $k\geq 0$ such that, if $x_{1},x_{2}$ are such that $c\notin
\lbrack x_{1},x_{2}]$ then $\forall y\in I:$%
\begin{equation*}
|G(x_{1},y)-G(x_{2},y)|\leq k\cdot |x_{1}-x_{2}|
\end{equation*}

\item $F|_{\gamma }$ is $\lambda $-Lipschitz with $\lambda <1$ (hence is
uniformly contracting) on each leaf $\gamma $.

\item $T:I\rightarrow I$ is\ onto and, piecewise monotonic, with two $C^{1}$
increasing branches on the intervals $[-\frac{1}{2},c)$,$(c,\frac{1}{2}]$
and $T^{\prime }>1$ where it is defined. Moreover $\underset{x\rightarrow
c^{-}}{\lim }T(x)=\frac{1}{2},\underset{x\rightarrow c^{+}}{\lim }T(x)=-%
\frac{1}{2},T(c)=-\frac{1}{2},$ $\underset{x\rightarrow c}{\lim }T^{\prime
}(x)=\infty $.

\item $\frac{1}{T^{\prime }}$ has bounded variation.
\end{enumerate}

then $(\Sigma ,F)$ has an unique invariant SRB measure which satisfies item
3 of Theorem \ref{resuno}.
\end{lemma}

We recall that the existence and the uniqueness of the SRB measure can be
obtained by the general arguments explained in Section \ref{sec:SBRfluxo}. To
proceed to prove the above statement, we need to introduce some concepts.

To deal with non normalized measures as the measures $\mu |\gamma $ on the
leaves are, we consider the following modification of the Wasserstein
distance: let  $b1lip(I)$ be the set of 1-Lipschitz functions on $I$ having $%
L_{\infty }$ norm less or equal than $1$ ($b1lip(I)=1lip(I)\cap \{g,\Vert
g\Vert _{\infty }\leq 1\}$).

Let us consider two finite measures $\mu ,\nu $ on $I$ and the distance%
\begin{equation*}
W_{1}^{0}(\mu ,\nu )=\sup_{g\in b1lip(I)}|\int g~d\mu -\int g~d\nu |.
\end{equation*}

\begin{remark}
\label{cont}We remark that choosing $g=1$ we obtain $W_{1}^{0}(\mu ,\nu
)\geq |\mu (I)-\nu (I)|$.
\end{remark}

Let us consider the space $M(I)$ of Borel finite measures over $I$ with the
distance $W_{1}^{0}$. Given a function $G:I\rightarrow (M(I),W_{1}^{0})$ we
define the variation of $G$ as follows: let $x_{1},...,x_{n}$ be an
increasing finite sequence in $I$ (which induces a subdivision in small
intervals) let $Sub$ be the set of such subdivisions. We define the
variation of $G$ as:%
\begin{eqnarray*}
Var(G,x_{1},...,x_{n}) &=&\sum_{i\leq n}W_{1}^{0}(G(x_{i}),G(x_{i+1})) \\
Var(G) &=&\sup_{(x_{i})\in Sub}Var(G,x_{1},...,x_{n}).
\end{eqnarray*}

We will consider the Lebesgue measure on the section $\Sigma $ and its
iterates by $F$. The strategy is to disintegrate along stable leaves and
estimate the variation of the induced function $I\rightarrow (M(I),W_{1}^{0})
$ proving that this is uniformly bounded. Let us precise this point: if $\mu $
is a finite measure on $\Sigma $, by disintegration this induces a function $%
G_{\mu }:I\rightarrow M(I)$ defined almost everywhere by%
\begin{equation*}
G_{\mu }(\gamma )=\mu |_{\gamma }.
\end{equation*}%
Suppose that $G_{\mu }$ is defined everywhere. The BV norm of $G_{\mu }$
will be an estimation of the regularity of $\mu $ on the $x$-axis. For example, the Lebesgue
measure on the square $\Sigma $ induces a function $G_{m}$ which is constant
everywhere and its value is the Lebesgue measure on the interval. The
variation in this case is obviously null. We remark that each iterate of the
Lebesgue measure by $F^{\ast }$ induces a $G_{F^{\ast n}(m)}$ which is
defined everywhere (see eq. \ref{iter} ). We will give an estimation of the
variation for these iterates in our system.

\begin{definition}
We say that a probability measure $\mu $ on $\Sigma $ is $K$-good if the function $G_{\mu }:I\rightarrow M(I),$ with $G_{\mu }(\gamma )=\mu |_{\gamma }$ as above is well defined and s.t. $Var(G_{\mu })\leq K$.
\end{definition}

\subparagraph{Some preliminary lemmata and remarks}

\begin{remark}
\label{11}We remark that if $\mu $ is $K$-good then $\sup_{\gamma }(\mu
|_{\gamma }(I))\leq 1+K.$
\end{remark}

\begin{proof}
Since $\mu $ is a probability measure then for some $\gamma ,$ $\mu
|_{\gamma }(I)\leq 1$ \ if for some $\xi $ it was $\mu |_{\xi }(I)>1+K$ then
by Remark \ref{cont} this would contradict  $Var(G_{\mu })\leq K$.
\end{proof}

This elementary remark about real sequences will be used in the following.

\begin{lemma}
\label{sequ}If a real sequence $a_{n}$ is such that $a_{n+1}\leq \lambda a_{n}+k$
\ for some $\lambda <1,k>0$, then%
\begin{equation*}
\sup (a_{n})\leq \max (a_{0},\frac{k}{1-\lambda })
\end{equation*}
\end{lemma}

\begin{proof}
If for some $m$, $a_{m}>\frac{k}{1-\lambda }$ then there is $\delta >0$ such
that $a_{m}=\frac{k+\delta }{1-\lambda }$. Hence, $a_{m+1}\leq \lambda \cdot 
\frac{k+\delta }{1-\lambda }+k=\frac{k+\lambda \delta }{1-\lambda }<a_{m}$.
Similarly $a_{n}\leq \frac{k}{1-\lambda }\implies a_{n+1}\leq \frac{k}{%
1-\lambda }$.
\end{proof}

The following is analogous to remark \ref{remark} for the distance $%
W_{1}^{0} $, and also follows by uniform contraction on stable leaves.

\begin{remark}
\label{remark2}Let $F$ be $\lambda $ contracting as above. Let us consider a
leaf $\gamma $ and two finite (non necessarily normalized) measures $\mu $, $%
\nu $ on it. Then 
\begin{equation*}
W_{1}^{0}(F^{\ast }(\mu ),F^{\ast }(\nu ))\leq |\mu (\gamma )-\nu (\gamma
)|+\lambda \cdot W_{1}^{0}(\mu ,\nu ).
\end{equation*}

\begin{proof}
If $g$ is in $b1lip$ on $F(\gamma )$ then $g(F(\ast ))$ is $\lambda $%
-Lipschitz on $\gamma $, moreover since $|g|\leq 1$ then $|g\circ F-\theta
|\leq \lambda $ for some $\theta \leq 1$. This implies that 
\begin{equation*}
|\int_{F(\gamma )}g~d(F^{\ast }\mu )-\int_{F(\gamma )}g~d(F^{\ast }\nu
)|=|\int_{\gamma }g\circ F~d\mu -\int_{\gamma }g\circ F~d\nu |\leq
\end{equation*}%
\begin{equation*}
\theta \cdot |\mu (I)-\nu (I)|+|\int_{\gamma }(g\circ F)-\theta ~d\mu
-\int_{\gamma }(g\circ F)-\theta ~d\nu |\leq
\end{equation*}%
\begin{equation*}
|\mu (I)-\nu (I)|+\lambda \cdot W_{1}^{0}(\mu ,\nu ).
\end{equation*}
\end{proof}
\end{remark}

Now we are ready to prove the main technical lemma estimating the regularity
of the iterates $F^{\ast n}(m)$. We will explicit the assumptions we need on 
$F$.

\begin{lemma}
\label{x1}Let $F(x,y)=(T(x),G(x,y))$ be a measurable map preserving the vertical
foliation such that:

\begin{enumerate}
\item There is $c\in I$  and $k\geq 0$ such that, if $x_{1},x_{2}$ are such that $c\notin
\lbrack x_{1},x_{2}]$ then $\forall y\in I:|G(x_{1},y)-G(x_{2},y)|\leq
k\cdot |x_{1}-x_{2}|$

\item $F|_{\gamma }$ is $\lambda $-Lipschitz with $\lambda <1$ uniformly on each
vertical leaf $\gamma $.
\end{enumerate}

Let $\gamma _{1}$ and $\gamma _{2}$ two close leaves with  $F^{-1}(\gamma _{1})=\{\alpha _{1},\alpha _{2}\},F^{-1}(\gamma _{2})=\{\beta
_{1},\beta _{2}\}$ and suppose that $T^{\prime }$ is defined at the points $%
\alpha _{i}$ and $\beta _{i}$ and at these points $T^{\prime }\geq 1$. Let $%
\mu _{0}$be a probability measure on $\Sigma $ such that $\mu _{0}|_{\gamma }
$ is defined for each $\gamma$ and 
\begin{equation*}
\mu _{0}|_{\gamma }(I)=\overline{f}_{0}(\gamma )
\end{equation*}%
for a bounded density function $\overline{f}_{0}$. Then 
\begin{gather*}
W_{1}^{0}(F^{\ast }(\mu _{0})|_{\gamma _{1}},F^{\ast }(\mu _{0})|_{\gamma
_{2}})\leq |\overline{f}_{0}(\alpha _{1})-\overline{f}_{0}(\beta
_{1})|+\lambda W_{1}^{0}(\mu _{0}|_{\alpha _{1}},\mu _{0}|_{\beta _{1}})+ \\
+|\overline{f}_{0}(\alpha _{2})-\overline{f}_{0}(\beta _{2})|+\lambda
W_{1}^{0}(\mu _{0}|_{\alpha _{2}},\mu _{0}|_{\beta _{2}}))+ \\
+2\cdot k\cdot \sup \overline{f}_{0}(|\alpha _{1}-\beta _{1}|+|\alpha
_{2}-\beta _{2}|)+\sup \overline{f}_{0}|\frac{1}{T^{\prime }(\alpha _{1})}-%
\frac{1}{T^{\prime }(\beta _{1})}|+\sup \overline{f}_{0}|\frac{1}{T^{\prime
}(\alpha _{2})}-\frac{1}{T^{\prime }(\beta _{2})}|.
\end{gather*}
\end{lemma}

\begin{proof}
Let $F^{\ast }(\mu _{0})|_{\gamma _{1}}$ be the restriction of $F^{\ast
}(\mu _{0})$ to the leaf $\gamma $. Remark that%
\begin{equation}
F^{\ast }(\mu _{0})|_{\gamma _{1}}=\frac{1}{T^{\prime }(\alpha _{1})}%
F_{\alpha _{1}}^{\ast }(\mu _{0}|_{\alpha _{1}})+\frac{1}{T^{\prime }(\alpha
_{2})}F_{\alpha _{2}}^{\ast }(\mu _{0}|_{\alpha _{2}})  \label{iter}
\end{equation}

where $F_{\alpha _{i}}:I\rightarrow I$ is given by $F_{\alpha _{i}}(y)=\pi
_{y}(F(y,\alpha _{i}))$ and%
\begin{equation*}
F^{\ast }(\mu _{0})|_{\gamma _{2}}=\frac{1}{T^{\prime }(\beta _{1})}F_{\beta
_{1}}^{\ast }(\mu _{0}|_{\beta _{1}})+\frac{1}{T^{\prime }(\beta _{2})}%
F_{\beta _{2}}^{\ast }(\mu _{0}|_{\beta _{2}})
\end{equation*}

with similar notation for $F_{\beta _{1}}$. Now the remaining part of the
proof is a (long) straightforward calculation:%
\begin{equation*}
W_{1}^{0}(F^{\ast }(\mu _{0})|_{\gamma _{1}},F^{\ast }(\mu _{0})|_{\gamma
_{2}})=\sup_{g\in b1lip}|\int g~d(F^{\ast }(\mu _{0})|_{\gamma _{1}})-\int
g~d(F^{\ast }(\mu _{0})|_{\gamma _{2}})|
\end{equation*}%
and%
\begin{eqnarray*}
\int g~d(F^{\ast }(\mu _{0})|_{\gamma _{1}}) &=&\int g~d(\frac{1}{T^{\prime
}(\alpha _{1})}F_{\alpha _{1}}^{\ast }(\mu _{0}|_{\alpha _{1}})+\frac{1}{%
T^{\prime }(\alpha _{2})}F_{\alpha _{2}}^{\ast }(\mu _{0}|_{\alpha _{2}})),
\\
\int g~d(F^{\ast }(\mu _{0})|_{\gamma _{2}}) &=&\int g~d(\frac{1}{T^{\prime
}(\beta _{1})}F_{\beta _{1}}^{\ast }(\mu _{0}|_{\beta _{1}})+\frac{1}{%
T^{\prime }(\beta _{2})}F_{\beta _{2}}^{\ast }(\mu _{0}|_{\beta _{2}})).
\end{eqnarray*}%
Let us estimate these two terms:%
\begin{eqnarray*}
\int gd(F^{\ast }(\mu _{0})|_{\gamma _{1}}) &=&\int gd(\frac{1}{T^{\prime
}(\alpha _{1})}F_{\alpha _{1}}^{\ast }(\mu _{0}|_{\alpha _{1}})+\frac{1}{%
T^{\prime }(\alpha _{2})}F_{\alpha _{2}}^{\ast }(\mu _{0}|_{\alpha _{2}}))=
\\
&=&\frac{1}{T^{\prime }(\alpha _{1})}\int g(F_{\alpha _{1}}^{{}}(y))d(\mu
_{0}|_{\alpha _{1}})+\frac{1}{T^{\prime }(\alpha _{2})}\int g(F_{\alpha
_{2}}^{{}}(y))d(\mu _{0}|_{\alpha _{2}})
\end{eqnarray*}%
and similarly%
\begin{equation*}
\int gd(F^{\ast }(\mu _{0})|_{\gamma _{2}})=\frac{1}{T^{\prime }(\beta _{1})}%
\int g(F_{\beta _{1}}^{{}}(y))d(\mu _{0}|_{\beta _{1}})+\frac{1}{T^{\prime
}(\beta _{2})}\int g(F_{\beta _{2}}^{{}}(y))d(\mu _{0}|_{\beta _{2}}).
\end{equation*}%
Hence

\begin{equation*}
|\int gd(F^{\ast }(\mu _{0})|_{\gamma _{1}})-\int gd(F^{\ast }(\mu
_{0})|_{\gamma _{2}})|=
\end{equation*}%
\begin{equation*}
|\frac{1}{T^{\prime }(\alpha _{1})}\int g(F_{\alpha _{1}}^{{}}(y))d(\mu
_{0}|_{\alpha _{1}})+\frac{1}{T^{\prime }(\alpha _{2})}\int g(F_{\alpha
_{2}}^{{}}(y))d(\mu _{0}|_{\alpha _{2}})
\end{equation*}%
\begin{equation*}
-\frac{1}{T^{\prime }(\beta _{1})}\int g(F_{\beta _{1}}^{{}}(y))d(\mu
_{0}|_{\beta _{1}})-\frac{1}{T^{\prime }(\beta _{2})}\int g(F_{\beta
_{2}}^{{}}(y))d(\mu _{0}|_{\beta _{2}})|.
\end{equation*}%
To estimate the last expression by the triangle inequality, let us add and
subtract 
\begin{equation*}
\frac{1}{T^{\prime }(\beta _{1})}\int g(F_{\beta _{1}}^{{}}(y))d(\mu
_{0}|_{\alpha _{1}})+\frac{1}{T^{\prime }(\beta _{2})}\int g(F_{\beta
_{2}}^{{}}(y))d(\mu _{0}|_{\alpha _{2}})
\end{equation*}%
obtaining 
\begin{equation*}
|\int gd(F^{\ast }(\mu _{0})|_{\gamma _{1}})-\int gd(F^{\ast }(\mu
_{0})|_{\gamma _{2}})|\leq |A|+|B|,
\end{equation*}%
where%
\begin{eqnarray*}
A &=&\frac{1}{T^{\prime }(\alpha _{1})}\int g(F_{\alpha _{1}}^{{}}(y))d(\mu
_{0}|_{\alpha _{1}})+\frac{1}{T^{\prime }(\alpha _{2})}\int g(F_{\alpha
_{2}}^{{}}(y))d(\mu _{0}|_{\alpha _{2}}) \\
&&-\frac{1}{T^{\prime }(\beta _{1})}\int g(F_{\beta _{1}}^{{}}(y))d(\mu
_{0}|_{\alpha _{1}})-\frac{1}{T^{\prime }(\beta _{2})}\int g(F_{\beta
_{2}}^{{}}(y))d(\mu _{0}|_{\alpha _{2}})
\end{eqnarray*}%
and%
\begin{eqnarray*}
B &=&\frac{1}{T^{\prime }(\beta _{1})}\int g(F_{\beta _{1}}^{{}}(y))d(\mu
_{0}|_{\alpha _{1}})+\frac{1}{T^{\prime }(\beta _{2})}\int g(F_{\beta
_{2}}^{{}}(y))d(\mu _{0}|_{\alpha _{2}}) \\
&&-\frac{1}{T^{\prime }(\beta _{1})}\int g(F_{\beta _{1}}^{{}}(y))d(\mu
_{0}|_{\beta _{1}})-\frac{1}{T^{\prime }(\beta _{2})}\int g(F_{\beta
_{2}}^{{}}(y))d(\mu _{0}|_{\beta _{2}})
\end{eqnarray*}%
\textbf{Estimation of A. }Now let us estimate $A$:%
\begin{eqnarray*}
|A| &\leq &|\frac{1}{T^{\prime }(\alpha _{1})}\int g(F_{\alpha
_{1}}^{{}}(y))d(\mu _{0}|_{\alpha _{1}})-\frac{1}{T^{\prime }(\beta _{1})}%
\int g(F_{\beta _{1}}^{{}}(y))d(\mu _{0}|_{\alpha _{1}})|+ \\
&&|\frac{1}{T^{\prime }(\alpha _{2})}\int g(F_{\alpha _{2}}^{{}}(y))d(\mu
_{0}|_{\alpha _{2}})-\frac{1}{T^{\prime }(\beta _{2})}\int g(F_{\beta
_{2}}^{{}}(y))d(\mu _{0}|_{\alpha _{2}})|=I+II
\end{eqnarray*}%
let us analyze the first term in the sum (the estimation of the other term
is similar)%
\begin{equation*}
I=|\frac{1}{T^{\prime }(\alpha _{1})}\int g(F_{\alpha _{1}}^{{}}(y))d(\mu
_{0}|_{\alpha _{1}})-\frac{1}{T^{\prime }(\beta _{1})}\int g(F_{\beta
_{1}}^{{}}(y))d(\mu _{0}|_{\alpha _{1}})|=
\end{equation*}%
\begin{equation*}
=|\int \frac{1}{T^{\prime }(\alpha _{1})}g(F_{\alpha _{1}}^{{}}(y))-\frac{1}{%
T^{\prime }(\beta _{1})}g(F_{\beta _{1}}^{{}}(y))~d(\mu _{0}|_{\alpha _{1}})|
\end{equation*}%
adding and subtracting $\frac{1}{T^{\prime }(\alpha _{1})}g(F_{\beta
_{1}}^{{}}(y))$ we obtain 
\begin{equation*}
|\int \frac{1}{T^{\prime }(\alpha _{1})}g(F_{\alpha _{1}}^{{}}(y))-\frac{1}{%
T^{\prime }(\beta _{1})}g(F_{\beta _{1}}^{{}}(y))+\frac{1}{T^{\prime
}(\alpha _{1})}g(F_{\beta _{1}}^{{}}(y))-\frac{1}{T^{\prime }(\alpha _{1})}%
g(F_{\beta _{1}}^{{}}(y))~d(\mu _{0}|_{\alpha _{1}})|\leq
\end{equation*}%
\begin{eqnarray*}
&\leq &|\int \frac{1}{T^{\prime }(\alpha _{1})}g(F_{\alpha _{1}}^{{}}(y))-%
\frac{1}{T^{\prime }(\alpha _{1})}g(F_{\beta _{1}}^{{}}(y))~d(\mu
_{0}|_{\alpha _{1}})|+ \\
&&+|\int \frac{1}{T^{\prime }(\alpha _{1})}g(F_{\beta _{1}}^{{}}(y))-\frac{1%
}{T^{\prime }(\beta _{1})}g(F_{\beta _{1}}^{{}}(y))~d(\mu _{0}|_{\alpha
_{1}})|.
\end{eqnarray*}%
Now, since $\overline{f}_{0}$ is bounded $\mu _{0}|_{\alpha _{1}}(I)\leq
\sup (\overline{f}_{0})$ and then 
\begin{equation*}
|\int \frac{1}{T^{\prime }(\alpha _{1})}g(F_{\beta _{1}}^{{}}(y))-\frac{1}{%
T^{\prime }(\beta _{1})}g(F_{\beta _{1}}^{{}}(y))~d(\mu _{0}|_{\alpha
_{1}})|\leq \sup \overline{f}_{0}|\frac{1}{T^{\prime }(\alpha _{1})}-\frac{1%
}{T^{\prime }(\beta _{1})}|.
\end{equation*}%
The other summand is 
\begin{equation*}
|\int \frac{1}{T^{\prime }(\alpha _{1})}g(F_{\alpha _{1}}^{{}}(y))-\frac{1}{%
T^{\prime }(\alpha _{1})}g(F_{\beta _{1}}^{{}}(y))~d(\mu _{0}|_{\alpha
_{1}})|\leq |\int g(F_{\alpha _{1}}^{{}}(y))-g(F_{\beta _{1}}^{{}}(y))~d(\mu
_{0}|_{\alpha _{1}})|
\end{equation*}%
By assumption (1) then $|F(y,\alpha _{1})-F(y,\beta _{1})|\leq k\cdot
|\alpha _{1}-\beta _{1}|$ and hence%
\begin{equation*}
|\int g(F_{\alpha _{1}}^{{}}(y))-g(F_{\beta _{1}}^{{}}(y))~d(\mu
_{0}|_{\alpha _{1}})|\leq k\cdot |\alpha _{1}-\beta _{1}|\cdot \sup 
\overline{f}_{0}.
\end{equation*}%
summarizing 
\begin{equation}
I\leq \sup \overline{f}_{0}|\frac{1}{T^{\prime }(\alpha _{1})}-\frac{1}{%
T^{\prime }(\beta _{1})}|+k\cdot |\alpha _{1}-\beta _{1}|\cdot \sup 
\overline{f}_{0}.
\end{equation}

Considering in the same way the summand $II$ in the expression of A, this
gives 
\begin{equation*}
|A|\leq k\cdot \sup \overline{f}_{0}(|\alpha _{1}-\beta _{1}|+|\alpha
_{2}-\beta _{2}|)+
\end{equation*}%
\begin{equation*}
\sup \overline{f}_{0}|\frac{1}{T^{\prime }(\alpha _{1})}-\frac{1}{T^{\prime
}(\beta _{1})}|+\sup \overline{f}_{0}|\frac{1}{T^{\prime }(\alpha _{2})}-%
\frac{1}{T^{\prime }(\beta _{2})}|.
\end{equation*}%
\textbf{Estimation of B.} The upper bound on B follows by contraction on
stable leaves. Indeed, 
\begin{eqnarray*}
|B| &\leq &|\frac{1}{T^{\prime }(\beta _{1})}\int g(F_{\beta
_{1}}^{{}}(y))d(\mu _{0}|_{\alpha _{1}})-\frac{1}{T^{\prime }(\beta _{1})}%
\int g(F_{\beta _{1}}^{{}}(y))d(\mu _{0}|_{\beta _{1}})|+ \\
&&+|\frac{1}{T^{\prime }(\beta _{2})}\int g(F_{\beta _{2}}^{{}}(y))d(\mu
_{0}|_{\alpha _{2}})-\frac{1}{T^{\prime }(\beta _{2})}\int g(F_{\beta
_{2}}^{{}}(y))d(\mu _{0}|_{\beta _{2}})|
\end{eqnarray*}%
now, since $F$ contracts all the leaves by a factor at least $\lambda $ by
Remark \ref{remark2} it holds%
\begin{eqnarray*}
|B| &\leq &\frac{1}{T^{\prime }(\beta _{1})}(\left\vert \mu _{0}|_{\alpha
_{1}}(I)-\mu _{0}|_{\beta _{1}}(I)\right\vert +\lambda W_{1}(\mu
_{0}|_{\alpha _{1}},\mu _{0}|_{\beta _{1}}))+ \\
&&+\frac{1}{T^{\prime }(\beta _{2})}(\left\vert \mu _{0}|_{\alpha
_{2}}(I)-\mu _{0}|_{\beta _{2}}(I)\right\vert +\lambda W_{1}(\mu
_{0}|_{\alpha _{2}},\mu _{0}|_{\beta _{2}})).
\end{eqnarray*}%
Summarizing, $\forall \,\,g\in b1lip$ 
\begin{gather}
\left\vert \int gd(F^{\ast }(\mu _{0})|_{\gamma _{1}})-\int gd(F^{\ast }(\mu
_{0})|_{\gamma _{2}})\right\vert \leq   \label{above} \\
\leq \frac{1}{T^{\prime }(\beta _{1})}(|\mu _{0}|_{\alpha _{1}}(I)-\mu
_{0}|_{\beta _{1}}(I)|+\lambda W_{1}(\mu _{0}|_{\alpha _{1}},\mu
_{0}|_{\beta _{1}}))+  \notag \\
+\frac{1}{T^{\prime }(\beta _{2})}(|\mu _{0}|_{\alpha _{2}}(I)-\mu
_{0}|_{\beta _{2}}(I)|+\lambda W_{1}(\mu _{0}|_{\alpha _{2}},\mu
_{0}|_{\beta _{2}}))+  \notag \\
+2k\sup \overline{f}_{0}(|\alpha _{1}-\beta _{1}|+|\alpha _{2}-\beta
_{2}|)+\sup \overline{f}_{0}|\frac{1}{T^{\prime }(\alpha _{1})}-\frac{1}{%
T^{\prime }(\beta _{1})}|+\sup \overline{f}_{0}|\frac{1}{T^{\prime }(\alpha
_{2})}-\frac{1}{T^{\prime }(\beta _{2})}|  \notag
\end{gather}%
finishing the proof.
\end{proof}

\begin{remark}
We remark that this last \ step in the proof (equations \ref{above} and
following) is the only one where the expansivity of $T$ is explicitly used.
In fact, an equivalent result can be obtained with the weaker assumption $%
\lambda (\underset{x\in I}{\inf }T^{\prime }(x))^{-1}<1$, instead of $T^{\prime
}>1$.
\end{remark}

A similar lemma holds for the case where the pre-image of $\gamma _{1}$ and $%
\gamma _{2}$ is only one leaf. The proof is similar to the previous one.

\begin{lemma}
\label{x2}Let $F:\Sigma \rightarrow \Sigma $ be as above, satisfying points
(1)--(3) of Lemma \ref{x1}. Let $\gamma _{1}$and $\gamma _{2}$ be two leaves
and suppose that $F^{-1}(\gamma _{1})=\{\alpha _{1}\},F^{-1}(\gamma
_{2})=\{\beta _{1}\}.$ Let us consider a probability measure $\mu _{0}$ on $%
\Sigma $ such that $\mu _{0}|\gamma (I)=\overline{f}_{0}(\gamma )$ for a
bounded function $\overline{f}_{0}$ , then%
\begin{gather*}
W_{1}^{0}(F^{\ast }(\mu _{0})|_{\gamma _{1}},F^{\ast }(\mu _{0})|_{\gamma
_{2}})\leq |\overline{f}_{0}(\alpha _{1})-\overline{f}_{0}(\beta
_{1})|+\lambda \cdot W_{1}(\mu _{0}|_{\alpha _{1}},\mu _{0}|_{\beta _{1}})+
\\
+2\cdot k\cdot \sup (\overline{f}_{0})(|\alpha _{1}-\beta _{1}|)+\sup (%
\overline{f}_{0})|\frac{1}{T^{\prime }(\alpha _{1})}-\frac{1}{T^{\prime
}(\beta _{1})}|.
\end{gather*}
\end{lemma}

The above lemmata give the following result, in the spirit of the Lasota
Yorke inequality (see in the following proof, eq. \ref{LY} and compare with 
\cite{LY73}, \cite{L04} e.g.) giving an upper bound on the variation of
iterates $F^{\ast n}(\mu _{0})$.

We recall that by the classical Lasota-Yorke inequalities, for piecewise
expanding maps of the interval, iterating a bounded variation density $g_{0}$
we get a sequence of uniformly bounded variation densities,%
\begin{equation}
Var(T^{\ast n}(g_{0}m))\leq C_{g_{0}}
\end{equation}%
where $C_{g_{0}}$ depends on $g_{0}$ and on the dynamics $T$.

\begin{theorem}
\label{bon}Let $F:\Sigma \rightarrow \Sigma $ be as above, satisfying
assumptions (1)--(4) of Lemma \ref{resdue}. Let $\mu _{n}=F^{\ast n}(\mu
_{0})$ where $\mu _{0}$ is $K$-good and has BV density on the $x$ axis, $%
\overline{f}_{0}$. Then, each $\mu _{n}$ is $K^{\prime }$-good, where $%
K^{\prime }=\max (K,\frac{3+C_{\overline{f}_{0}}+(C_{\overline{f}_{0}}+1)Var(%
\frac{1}{T^{\prime }})+2k(C_{\overline{f}_{0}}+1)}{1-\lambda })$.
\end{theorem}

\begin{proof}
Let us consider a subdivision $\gamma _{1},...,\gamma _{n}$ made of small
intervals and set $s_{i}=T^{-1}([\gamma _{i},\gamma _{i+1}))$. If we are in
the case of Lemma \ref{x1} $s_{i}$ consists of two small intervals, if we
are in the case of Lemma \ref{x2} $s_{i}$ consists of one small interval and
in the remaining case we have one small interval and an interval of the type 
$(\alpha _{2},\frac{1}{2})$ or $(-\frac{1}{2},\beta _{1})$ (this can happen
only in two intervals of the subdivision containing the points $T(-\frac{1}{2})$ and $T(\frac{1}{2})$ ) . The endpoints of all these pre-image intervals
$(s_{i})_{i\in (1,...,n)}$ constitute another subdivision $\gamma _{1}^{\ast
},...,\gamma _{m}^{\ast }$ of $I$.

Let us estimate the variation of $\mu _{1}=F^{\ast }(\mu _{0})$ on the
subdivision $\gamma _{1},...,\gamma _{n}$. Let us suppose that the intervals
of $\gamma _{1},...,\gamma _{n}$ which are of the third type are $(\gamma
_{j_{1}},\gamma _{j_{1}+1})$ and $(\gamma _{j_{2}},\gamma _{j_{2}+1})$. In
this case we bound trivially from above the variation: $W_{1}^{0}(\mu
_{1}|_{\gamma _{j_{i}}},\mu _{1}|_{\gamma _{j_{i}+1}})\leq \sup \overline{f}%
_{0}$ (for $i=1,2$). Lemma \ref{x1} and Lemma \ref{x2} imply%
\begin{equation*}
Var(G_{\mu _{1}},\gamma _{1},...,\gamma _{n})=
\end{equation*}%
\begin{equation*}
\sum_{i\leq n}W_{1}(\mu _{1}|_{\gamma },\mu _{1}|_{\gamma _{i+1}})\leq 2\sup 
\overline{f}_{0}+\sum_{i\leq m}|\overline{f}_{0}(\gamma _{i}^{\ast })-%
\overline{f}_{0}(\gamma _{i+1}^{\ast })|+\lambda W_{1}(\mu _{0}|_{\gamma
_{i}^{\ast }},\mu _{0}|_{\gamma _{i+1}^{\ast }}))+
\end{equation*}%
\begin{equation*}
+\sum_{i\leq m}2k\sup \overline{f}_{0}(|\gamma _{i}^{\ast }-\gamma
_{i+1}^{\ast }|)+\sup \overline{f}_{0}|\frac{1}{T^{\prime }(\gamma
_{i}^{\ast })}-\frac{1}{T^{\prime }(\gamma _{i+1}^{\ast })}|.
\end{equation*}%
Hence%
\begin{equation*}
Var(G_{\mu _{1}},\gamma _{1},...,\gamma _{n})\leq 2\sup \overline{f}_{0}+Var(%
\overline{f}_{0})+\sup \overline{f}_{0}Var(\frac{1}{T^{\prime }})+\sup 
\overline{f}_{0}2k+\lambda Var(G_{\mu _{0}})
\end{equation*}%
and we conclude that%
\begin{equation}
Var(G_{\mu _{1}})\leq 2\sup \overline{f}_{0}+Var(\overline{f}_{0})+\sup 
\overline{f}_{0}Var(\frac{1}{T^{\prime }})+\sup \overline{f}_{0}2k+\lambda
Var(G_{\mu _{0}}).  \label{LY}
\end{equation}%
If $\overline{f}_{i}$are the marginals of $\mu _{i}$ on the $x$-axis, then as recalled before 
$Var(\overline{f}_{i})\leq C_{\overline{f}_{0}}$. This allows to iterate the
above inequality and obtain, by Lemma \ref{sequ} 
\begin{equation*}
\sup_{i}(Var(G_{\mu _{i}}))\leq \max (Var(G_{\mu _{0}}),\frac{2+3C_{%
\overline{f}_{0}}+(C_{\overline{f}_{0}}+1)Var(\frac{1}{T^{\prime }})+2k(C_{%
\overline{f}_{0}}+1)}{1-\lambda }),
\end{equation*}%
(remark that $\sup \overline{f}_{0}\leq Var(\overline{f}_{0})+1$) finishing
the proof.
\end{proof}

If $\mu $ is a good measure, the measure $f\mu $ associated to a Lipschitz
observable $f$ is also a good measure:

\begin{lemma}
\label{798}If $\mu _{n}$ is a sequence of $K-$good measures on $\Sigma $ and 
$\nu _{n}=f\mu _{n}$ with $f$ be $\ell $-Lipschitz and $\Vert f\Vert
_{\infty }\leq \ell $ then each $\nu _{n}$ is a $(3\ell K+\ell )$-good
measure.
\end{lemma}

\begin{proof}
let $\gamma _{1}$ and $\gamma _{2}$ be two close leaves%
\begin{equation*}
W_{1}^{0}(\nu _{n}|_{\gamma _{1}},\nu _{n}|_{\gamma _{2}})=\sup_{g\in
b1lip}\left\vert \int_{\gamma _{1}}g(\ast )f(\ast ,\gamma _{1})d(\mu
_{n}|_{\gamma _{1}})-\int_{\gamma _{2}}g(\ast )f(\ast ,\gamma _{2})d(\mu
_{n}|_{\gamma _{2}})\right\vert
\end{equation*}%
we recall that $|g|\leq 1$, hence%
\begin{gather*}
\left\vert \int g(\ast )f(\ast ,\gamma _{1})~d(\mu |_{\gamma _{1}})-\int
g(\ast )f(\ast ,\gamma _{2})~d(\mu |_{\gamma _{2}})\right\vert \leq \\
\leq \left\vert \int g(\ast )f(\ast ,\gamma _{1})~d(\mu |_{\gamma
_{1}})-\int g(\ast )f(\ast ,\gamma _{1})~d(\mu |_{\gamma _{2}})\right\vert +
\\
+\left\vert \int g(\ast )f(\ast ,\gamma _{1})~d(\mu |_{\gamma _{2}})-\int
g(\ast )f(\ast ,\gamma _{2})~d(\mu |_{\gamma _{2}})\right\vert \leq
\end{gather*}%
\begin{equation*}
\leq 2\ell \cdot W_{1}^{0}(\mu _{n}|_{\gamma _{1}},\mu _{n}|_{\gamma
_{2}})+\ell |\gamma _{1}-\gamma _{2}|\sup_{\gamma }(\mu |_{\gamma }(I))
\end{equation*}%
hence $Var(G_{\nu _{n}})\leq 2\ell K+\ell (K+1).$
\end{proof}

\begin{remark}
\label{123}If $\mu _{n}$ is $K$-good for each $n$ and $\overline{g}%
_{n}:I\rightarrow I$ is the marginal, such that%
\begin{equation*}
\overline{g}_{n}(\gamma )=\mu _{n}|_{\gamma }(I),
\end{equation*}%
since $|\overline{g}_{n}(\gamma _{1})-\overline{g}_{n}(\gamma _{2})|\leq
W_{1}^{0}(\mu _{n}|_{\gamma _{1}},\mu _{n}|_{\gamma _{2}})$ then it holds 
\begin{equation}
Var(\overline{g}_{n})\leq K
\end{equation}%
for each $n$.
\end{remark}

\begin{remark}
\label{456}If $\mu _{n}\rightarrow \mu $ and $\nu _{n}=f\mu _{n}$ , $\nu
=f\mu $ with $f$ be $\ell $-Lipschitz then $\nu _{n}\rightarrow \nu .$(this
is easily obtained because $\int hf~d\mu _{n}\rightarrow \int hf~d\mu $, for
each continuous $h$ since $hf$ is continuous).
\end{remark}

We are finally ready to end the proof of the main proposition of the section.

\begin{proof}
(of Lemma \ref{resdue}) We prove that $\overline{f}$ as defined at item 3 of
Theorem \ref{resuno} has bounded variation and $Var(\overline{f})\leq 3\ell
K^{\prime }+\ell $, where $\ell $ is the Lipschitz constant of $f$ and $K^{\prime }$ does not depend on $f$. Let $\mu _{n}=F^{\ast n}(m)$ be the sequence of iterates of the Lebesgue measure $m$.
By Theorem \ref{bon} these are $K^{\prime }$-good. 
By Proposition \ref{prod}, Equation \ref{l1bv} and uniform contraction on unstable leaves it follows  $\mu _{n}\rightarrow \mu $ in the weak topology. 
Then for each continuous $h$ it holds $\mu
_{n}(h)\rightarrow \mu (h)$. In particular this holds for the functions
which are constant on each contracting leaf. Let $h$ be such a function.
Then $\int_{\Sigma }h~d\mu _{n}=\int_{I}h\overline{g}_{n}dx$ where $%
\overline{g}_{n}(x)=\mu _{n}|_{\gamma _{x}}(I)$ are the densities of $\mu
_{n}$ on the $x$ axis as in Remark \ref{123}.

Let $f$ be $\ell $-Lipschitz, $\nu _{n}=f\mu _{n}$ and $\nu =f\mu $ as
required by Lemma \ref{resdue}. Since $h$ is constant along the leaves,
again $\int_{\Sigma }h~d\nu =\int_{I}h\overline{f}dx$ and $\int_{\Sigma
}h~d\nu _{n}=\int_{I}h\overline{f}_{n}dx$ where $\overline{f}_{n}(\gamma
)=\int_{\gamma }f~d(\mu _{n}|_{\gamma })$ as above. By Remark \ref{456} 
\begin{equation*}
\int_{\Sigma }h~d\nu _{n}\rightarrow \int_{\Sigma }h~d\nu 
\end{equation*}%
hence 
\begin{equation*}
\int_{I}h\overline{f}_{n}dx\rightarrow \int_{I}h\overline{f}dx.
\end{equation*}

We have to prove that $\overline{f}$ is BV. By Lemma \ref{798} the measures $%
\nu _{n}$ are $(3\ell K^{\prime }+\ell )$-good. Now by Remark \ref{123}, $%
Var(\overline{f}_{n})\leq 3\ell K^{\prime }+\ell $ . By the Helly theorem
there is a sub-sequence $\overline{f}_{n_{i}}$ converging in the $L^{1}$
norm to some bounded variation function $\tilde{f}$ such that $Var(\tilde{f}%
)\leq 3\ell K^{\prime }+\ell $.

Hence $\int h\overline{f}_{n_{i}}dx\rightarrow \int h\tilde{f}dx$ for each $%
h $ as above and so $\int h\overline{f}dx=\int h\tilde{f}dx$ for each
continuous $h$ and then this implies that they coincide a.e.. Hence $%
\overline{f}$ can be supposed to be BV and having $Var(\overline{f})\leq
3\ell K^{\prime }+\ell $ .
\end{proof}



\section{Appendix II: Exact dimensionality}
\label{sec:steinberg}
\label{sec:muexata}

In several of the above results we used the local dimension of the system at certain points.
In this section we  recall a result of of Steinberger (\cite{S00}) about the local dimension of Lorenz like systems and prove that for the geometric Lorenz system  the local dimension is defined at almost every point. Let us recall the assumptions used in \cite{S00} .

Let us consider a map $F:[0,1]^2\to [0,1]^2$, $F(x,y)=(T(x),g(x,y))$ where
\begin{enumerate}
 \item[(1)] $T:[0,1]\to [0,1]$ is piecewise monotonic.
  This means that there are $c_i\in [0,1]$ for $0\leq i\leq N$ with
$0=c_0< \cdots < c_N=1$ such that $T|(c_i,c_{I+1})$ is continuous and
monotone for $0\leq i < N$. Furthermore, for $0\leq i< N$, $T|(c_i,c_{i+1})$ is $C^1$
and that $\inf_{x\in \cS}|T'(x)|> 0$ holds where $\cP=[0,1]\setminus \cup_{0\leq i < N}c_i$.
\item[(2)] $g:[0,1]^2\to (0,1)$ is $C^1$ on $\cP\times [0,1]$. Furthermore,
$\sup|\partial g/\partial x| < \infty$, $\sup|\partial g/\partial y| < 1$ and
$|(\partial g/\partial y)(x,y)| > 0$ for $(x,y)\in \cP\times [0,1]$.
\item[(3)] $F((c_i,c_{i+1})\times [0,1])\cap F((c_i,c_{i+1})\times [0,1])=\emptyset$
for distinct $i, j$ with $0\leq i, j < N$.
\end{enumerate}

%

Now consider the projection $\pi_{x}:I^2\to I$, set $\cV=\{(-1/2,0), (0,1/2)\}$ and $\cV_k=\bigvee_{i=0}^{k}f^{-i}\cV$, which is a partition of
$E=\bigcap_{i=0}^\infty H^{-1}(I\setminus \{0\})$ into open intervals.
For $x\in E$ let $J_k(x)$ be the unique element of $\cV_k$ which contains $x$.
We say that $\cV$ {\em is a generator} if the length of the intervals
$J_k(x)$ tends to zero  for $n \to \infty$  for any  given $x$.
Set
$$\psi(x,y)=\log|T'(x)| \quad \mbox{and} \quad \varphi(x,y)=-\log|(\partial g/\partial y)(x,y)|.$$

The result of Steinberger that we shall use is the following

\begin{theorem}\cite[Theorem 1]{S00}
\label{th:Steinberg} Let $F$ be a two-dimensional map as above and $\mu$ an ergodic, $F$-invariant probability measure on $I^2$ with the entropy $h_{\mu}(F)>0$.
Suppose $\cV$ is a generator, $\int\psi\cdot d\mu_F < \infty$ and $0<\int \varphi d\mu_F < \infty$.
If the maps $y\mapsto \psi(x,y)$ are uniformly equicontinuous for $x \in I\setminus\{0\}$
and $1/|f'|$ is BV then
$$
d_{\mu}(x,y)=h_{\mu}(F)\big(\frac{1}{\int\psi\cdot d\mu} + \frac{1}{\int \varphi\cdot d\mu}\big)
$$
for $\mu$-almost all $(x,y) \in I^2$.
\end{theorem}

Now we verify that the Lorenz geometric system as defined in Section
\ref{sec:SBRfluxo} is exact dimensional. First we observe that for
the first return map $F:\Sigma\setminus \Gamma\to \Sigma$ associated
to the Lorenz geometric flow  its entropy $h_\mu(F) > 0$, see
\cite[pp.188]{APbook,APPV08}. Next, equations (\ref{gy}),
(\ref{gx}), and the properties of $f_{Lo}$ described in Subsections
\ref{asegundacoordenada} and \ref{sec:propert-one-dimens} guaranty
that $F=(f_{Lo},g_{Lo})$ is a two-dimensional transformation
satisfying the above points (1--3). So, all we need to prove that
$(\Sigma,F,d\mu_F)$ is exact dimensional is to verify that
$F(x,y)=(f_{Lo}(x),g_{Lo}(x,y))$ satisfies the hypothesis of
Theorem \ref{th:Steinberg}. For this, let
$$\psi(x,y)=\log|f_{Lo}'(x)| \quad \mbox{and} \quad \varphi(x,y)=-\log|(\partial g_{Lo}/\partial y)(x,y)|.$$
Then the following result holds:

\begin{proposition}
\label{pro:steinberg}
 For $q=(x,y) \in \Sigma^*$, let $\varphi(q)=-\log|\partial g_{Lo}/\partial y (q)|$
and $\psi(q)=\log|f_{Lo}'(x)|$. Then
\begin{enumerate}
 \item $\int \varphi d\mu_F < \infty$,
\item $0< \int \psi d\mu_F < \infty$, and
\item the maps $y \mapsto \varphi(x,y)$ are uniformly equicontinuous
for $x\in I\setminus\{0\}$.
\end{enumerate}
where $\mu_F$ is the invariant ergodic SRB measure described in Subsection \ref{sec:SBRfluxo}.
\end{proposition}
\begin{proof}

Given $q=(x,y)\in [-1/2,1/2]^2$,  we provide the calculations for $x>0$, the other case
being analogous.

By equation (\ref{eq:derivadaF}) we have
\begin{align*}
 DF(x,y)= \left(
\begin{array}{cccc}
    \partial_xf_{Lo}  & \partial_yf_{Lo}   \\
    \partial_xg_{Lo}  & \partial_yg_{Lo}
\end{array}
\right)
=
\left(
\begin{array}{cccc}
   M\cdot \alpha\cdot x^{(\alpha-1)} &  0
  \\
  \sigma\cdot \beta\cdot
  yx^{(\beta-\alpha)}
  & \sigma x^\beta
\end{array}
\right).
\end{align*}

\noindent {\em Proof of (1):}
By the expression above we have $\partial g_{Lo}/ \partial y(q)=\sigma\cdot x^\beta$ and so
 $\log|\partial g_{Lo}/ \partial y(q)|=\log|\sigma\cdot x^\beta|$ does not depend on $y$.
Since the measure $\mu_F$ is constant at each leaf $\ell \in \cFŝ$ and the projection
of $\mu_F$ on the $x$-axis, $\mu_{f_{Lo}}$, is absolutely continuous with respect to Lebesgue
measure (and even has a finite density), see Proposition \ref{prop:densidadeBVdef}, we immediately conclude that
$$
\int \log|\partial g_{Lo}/ \partial y(q)| d\mu_F < \infty.
$$
proving (1).

\noindent {\em Proof of (2):}
Again from the expression for $DF(x,y)$ above we get $f_{Lo}'(x)=M\cdot \alpha\cdot x^{(\alpha-1)}$, recall $0< \alpha<1$.
Hence, for $0< x < 1/2$, $\log(f_{Lo}'(x))=\log(M\cdot \alpha\cdot x^{(\alpha-1)})$.
Thus
$$ 0< \int \log(f_{Lo}'(x)) d\mu_F \leq K_0 + (\alpha - 1) [x\cdot \log(x) - x] \leq K_0 + (\alpha -1) K_1,
$$
proving (2).

\noindent {\em Proof of (3)}

Note that $\varphi(x,y)=-\log|(\partial g_{Lo}/ \partial y)(x,y)|= \log (\sigma) + \beta\cdot \log(|x|)$ and
so the maps $y \mapsto \varphi(x,y)$ are obviously uniformly equicontinuous for $x\neq 0$.

All together finishes the proof of Proposition \ref{pro:steinberg} establishing that $\mu_F $ is exact dimensional.
\end{proof}

\end{document}